%% file: increasing-compacted.tex
\documentclass[reqno]{amsart}

\usepackage{mathrsfs}
\usepackage{graphicx} 
\usepackage{a4wide} 
\usepackage[latin1]{inputenc}
\usepackage[english]{babel}
\usepackage[T1]{fontenc}
\usepackage{xcolor}
\usepackage{tikz}
\usepackage{url}

\DeclareGraphicsExtensions{.png,.pdf,.jpg}
\usepackage{wrapfig}

\usepackage[width=0.9\textwidth]{caption}

\renewcommand{\(}{\left(}
\renewcommand{\)}{\right)}
\newtheorem{thm}{Theorem}
\newtheorem{prop}{Proposition}
\newtheorem{lem}{Lemma}
\newtheorem{cor}{Corollary\!\!}

\newtheorem{ncor}{Corollary}
\newtheorem{conj}{Conjecture}

\theoremstyle{definition}

\newtheorem{df}{Definition}
\newtheorem{ex}{Example}

\theoremstyle{remark}

\newtheorem{rem}{Remark\!\!}
 
\newtheorem{rems}{Remarks\!\!}

\newtheorem{nrem}{Remark}

\newcommand{\Ord}[1]{{\mathcal O}\left(#1\right)}  
\newcommand{\rcp}[1]{\frac{1}{#1}} 
\newcommand{\E}{{\mathbb E}}

\newcommand{\bth}{\begin{theo}} 
\newcommand{\eth}{\end{theo}} 
\newcommand{\bl}{\begin{lem}} 
\newcommand{\el}{\end{lem}} 
\newcommand{\bp}{\begin{prop}} 
\newcommand{\ep}{\end{prop}} 
\newcommand{\bdf}{\begin{df}} 
\newcommand{\edf}{\end{df}} 
\newcommand{\brem}{\begin{rem}} 
\newcommand{\erem}{\end{rem}} 
\newcommand{\brems}{\begin{rems}} 
\newcommand{\erems}{\end{rems}} 
\newcommand{\bnrem}{\begin{nrem}} 
\newcommand{\enrem}{\end{nrem}} 
\newcommand{\bex}{\begin{ex}} 
\newcommand{\eex}{\end{ex}} 
\newcommand{\bcor}{\begin{cor}} 
\newcommand{\ecor}{\end{cor}} 
\newcommand{\bncor}{\begin{ncor}} 
\newcommand{\encor}{\end{ncor}} 
\newcommand{\bpf}{\begin{proof}} 
\newcommand{\epf}{\end{proof}}

\newcommand{\nti}{{n\to\infty}}

\newcommand{\dt}{\,\mathrm{d}t}

\newcommand{\dx}{\,\mathrm{d}x}
\newcommand{\dy}{\,\mathrm{d}y}
\newcommand{\dz}{\,\mathrm{d}z}
\newcommand{\dv}{\,\mathrm{d}v}

 \DeclareMathOperator*{\Set}{\mbox{\sc Set}}

 \newcommand{\T}{\mathcal{T}} 
 \newcommand{\Z}{\mathcal{Z}} 
 
\begin{document}

\title{Compaction for two models of logarithmic-depth trees:
Analysis and Experiments} 
\thanks{This work
 was partially supported by the \textsc{anr} project \textsc{Metaconc} ANR-15-CE40-0014,
 by the \textsc{phc} \#~39454SF, by the ÖAD grant FR04/2018 
 and by the Austrian Science Foundation (FWF), grant SFB F50-03.}

\author[O. Bodini]{Olivier Bodini}
\address{Olivier Bodini \and Mehdi Naima. Université Sorbonne Paris Nord, Laboratoire d'Informatique de Paris Nord, CNRS,
UMR 7030, F-93430, Villetaneuse, France.}
\email{\{Olivier.Bodini, Mehdi.Naima\}@lipn.univ-paris13.fr}

\author[A. Genitrini]{Antoine Genitrini}
\address{Antoine Genitrini. Sorbonne Universit\'e, CNRS,
Laboratoire d'Informatique de Paris 6 -LIP6- UMR 7606, F-75005 Paris, France.
}
\email{Antoine.Genitrini@lip6.fr}

\author[B. Gittenberger]{Bernhard Gittenberger}
\address{Bernhard Gittenberger and I. Larcher. Department of Discrete Mathematics and Geometry,
	Technische Universit\"at Wien, Wiedner Hauptstra\ss e 8-10/104, 1040 Wien, Austria.}
\email{\{Gittenberger, Larcher\}@dmg.tuwien.ac.at}

\author[I. Larcher]{Isabella Larcher}

\author[M. Naima]{Mehdi Naima}

\date{\today}

\begin{abstract} 
We are interested in the quantitative analysis of the compaction ratio for two classical families
of trees: recursive trees and plane binary increasing trees. These families are typical
representatives of tree models with a small depth. Once a tree of size $n$ is compacted by keeping
only one occurrence of all fringe subtrees appearing in the tree the resulting graph contains only
$O(n / \ln n)$ nodes.  This result must be compared to classical results of compaction in the
families of simply generated trees, where the analogous result states that the compacted structure
is of size of order $n / \sqrt{\ln n}$.  The result about the plane binary increasing trees has
already been proved, but we propose a new and generic approach to get the result.  Finally, an
experimental study is presented, based on a prototype implementation of compacted binary search
trees that are modeled by plane binary increasing trees.\\

\end{abstract}

\maketitle


\input{sections/intro.tex}

\input{sections/recursive.tex}

\input{sections/plane.tex}

\input{sections/datastruct.tex}

\input{sections/concl.tex}

\section*{Acknowledgments} 
The authors thank the anonymous referees for pointing out several references, but also for their
comments and suggested improvements. In particular, we express our gratitude to one of the
referees who pointed out a subtle error and several smaller ones. All these persistent remarks 
have greatly increased the quality of the paper. 

\bibliographystyle{plain}
\bibliography{sections/biblio}

\end{document}

%% file: sections/intro.tex
\section{Introduction} 
	\label{sec:intro}

\begin{wrapfigure}[24]{r}{6.8cm}
    \begin{center}
    	\vspace*{-0.2cm}
        \includegraphics[scale=0.08]{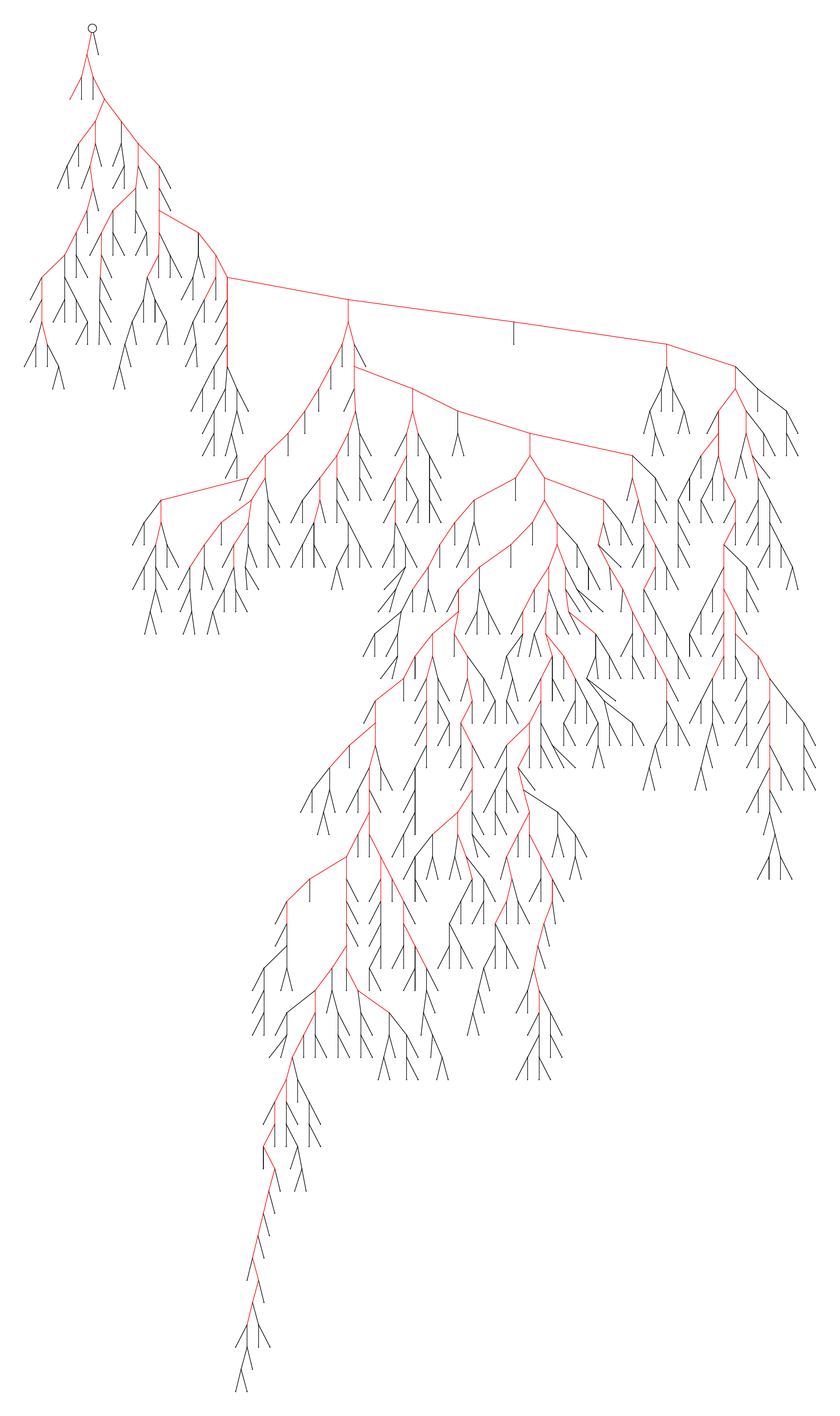}
        \caption{\label{fig:simpliTree}
            A uniformly sampled plane binary tree with 500 internal nodes:
            Black fringe subtrees are removed by the compaction process.
            The red part (that is what remains of the tree after pruning
the black fringe subtrees) is of size~250.}
    \end{center}
\end{wrapfigure}
Tree-shape data structures are omnipresent in computer science.
The syntax structure of a program is a tree, symbolic expressions in computer algebra systems
have a tree structure. Syntax trees arise in the context of parsing, XML data structures
are also built on trees. However in order to reduce redundancy in the storage, usually
an algorithmic step called the \emph{common subexpression recognition} is run to identify
identical fringe subtrees (\emph{i.e.} a node and all its descendants)
so that only one occurrence is stored and all other are replaced
by pointers to the first one. Thus the trees are then replaced by directed acyclic graphs.
In the context of tree compaction several studies attempt to quantitatively analyze the process of
compaction. We mention here in particular two important research lines about compaction properties.

The first line occurs in the context of information theory and data compression studies.
There researchers are interested in designing compression algorithms for advanced data structures. 
One of the main parameters of interest is the entropy of the data structure: it represents an
optimal lower bound on the average number of bits required to represent the data 
structure: see for example~\cite{cover2006elements} for an introduction to the subject.
For trees, the entropy of some models of plane trees\footnote{Plane trees are such that
the descendants of a node are ordered contrary to non-plane trees where the descendants
are seen as a set instead of a sequence of subtrees.} has been studied in particular in
\cite{matu2018,cisz17,golkebiewski2018entropy}. 

An analysis of a model of non-plane binary trees has been presented in~\cite{cisz17}. 
The authors focus on the number of symmetry nodes (internal nodes having two isomorphic subtrees
as children) and its relation with R\'enyi entropy. In all investigations of that kind, the
probability distribution used for the tree model is central. The aforementioned work~\cite{cisz17} 
is focusing on a growing tree model that is also seen as the classical binary search tree
distribution model. Likewise, it can be rephrased as the binary increasing tree model we will deal
with in Section~\ref{sec:plane}, as it was already pointed out in \cite{BFS92}. 
We are, however, interested in different aspects of these trees
(see below for more details).

The second line of research has been started by the seminal paper of Flajolet \emph{et al}~\cite{FSS90}.
In this paper the authors consider the compaction ratio
of classical binary trees compared with their corresponding compacted structures. They prove,
starting from a large binary tree of size $n$ (containing $n$ nodes) and then compacting it,
that the average size of the compacted result is $\alpha n / \sqrt{\ln n}$ with a computable constant $\alpha$.
In the end of the paper the authors finally state that their analysis is fully adapted
to all families of simply generated trees as defined by Meir and Moon in their fundamental
paper~\cite{MM78} and thus for all kinds of tree structures we mentioned above as examples, we get
the same kind of ratio for the compaction. In Figure~\ref{fig:simpliTree}
we have represented a uniformly sampled binary tree with 500 internal nodes. If we compact it
then all the fringe subtrees 
in black are removed and only the red structure is kept 
with addition of several pointers (that are not represented in the figure). The remaining red tree is of size 250.
We recall that in the context of simply generated trees of size $n$,
the typical depth is of order $\sqrt{n}$ (this is the case for the binary trees).
Bousquet-Mélou \emph{et al.}~\cite{BMLMN15} present the complete proof
for the compaction quantitative analysis of simply generated tree families and apply it experimentally
on XML-trees. Finally, in~\cite{RW15} the authors
are interested in the number of fringe subtrees with at least $r$ occurrences in a random simply generated tree.
This approach is an extension of the previous results where it was dealt
with subtrees appearing at least once (thus for $r=1$).\\

But there are also several other kinds of tree structures that cannot be modeled
through the concept of simply generated trees. In particular, we have in mind
all structures used for searching, and thus usually with a small depth of order $\ln n$
for a whole structure of size $n$. The classical binary search trees (\textsc{bst}), 
red-black trees or AVL trees belong to these families. 
But we can also point out priority heaps like binary or binomial heaps.
The reader can refer for example to Knuth's book~\cite{Knuth98} for details
about all these structures. In this context, all nodes contain different labels (or information)
and thus the compaction process as described before has no effect (no two subtrees are identical due to the labeling).
But, if we remove the labels from the nodes, then a tree structure remains 
whose typical depth is of order $\ln n$ for $n$ nodes. Hence we can compact the tree structure.
\begin{figure}[h]
	\includegraphics[width=0.9\textwidth]{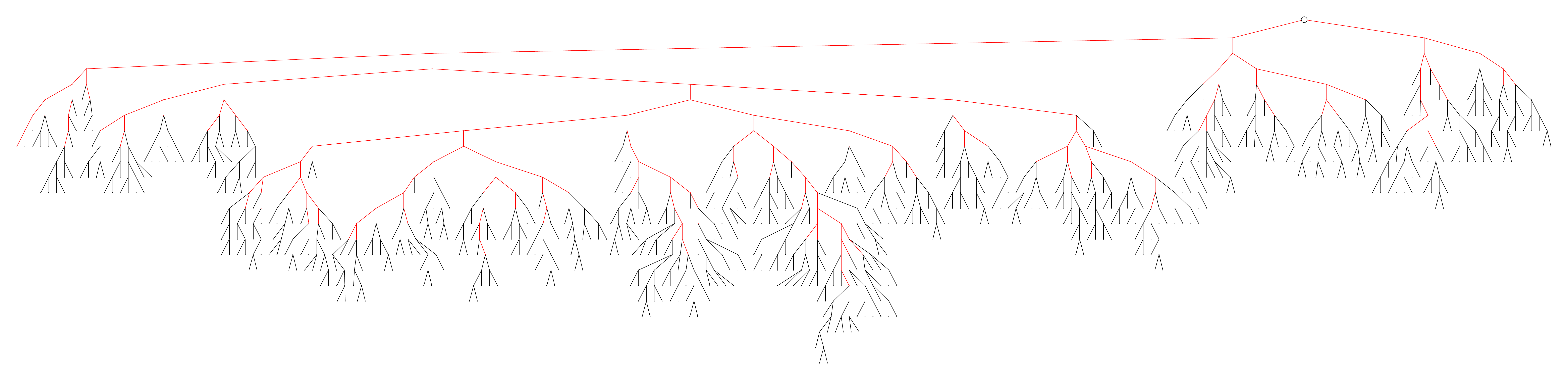}
    \caption{\label{fig:searchTree}
		A uniformly sampled (plane) binary search tree structure with 500 internal nodes: Black fringe subtrees are removed
		by the compaction. The red part is only of size~172.}
\end{figure}
In Figure~\ref{fig:searchTree} we have depicted a binary search tree structure of size 500.
Once the structure is compacted, it remains a tree with 172 nodes (represented in red).

Our study focuses on the number of non-isomorphic subtrees in a tree and this
corresponds also to the size of the compacted tree (also called minimal DAG
representation in \cite{zhang2013redundancy}). This parameter is different
from the study of symmetry nodes mentioned above (see~\cite{cisz17}), since there 
symmetries happen if an internal node has two isomorphic children (a local symmetry)
whereas the number of non-isomorphic subtrees of a tree is capturing a global symmetry. Using the 
results in \cite{cisz17} to design and analyze a data compression algorithm leads to constant
compression rate on average, as was already shown in \cite{FGM97}. In our case, we gain on average 
at least a logarithmic factor.  

For both investigations, a Riccati-like functional-differential equation must be analyzed. But,
not only the equations in \cite{cisz17} and in Section~\ref{sec:plane} are different, but the
global nature of the parameter of our interest is reflected by the need of uniform asymptotics,
which required a delicate singularity analysis.

In this paper, we analyze the underlying \emph{unlabeled tree structure} of
a plane and a non-plane model of increasingly labeled trees, namely
increasing binary trees and recursive trees. For these two models of trees picking
a tree uniformly at random and erasing the labels from it gives an unlabeled
plane binary tree or an unlabeled non-plane general tree (also called P\'olya tree).
However, for each model the probability distribution of the
resulting unlabeled tree is non-uniform. The distribution on plane binary trees we use
is the same as the one of \cite{cisz17,matu2018}. Even if the analyzed parameters
are not the same, for all such studies the mathematical tools are based on differential equation analyses 
due to the underlying distribution on trees.

Finally, another way to reach the non-uniform distribution is as a
very simple natural evolution process. First let us mention the
plane binary tree model: start with a single node, and at each step select
randomly one of the
leaves (external node) and replace with a binary node. While for P\'olya trees, 
start with a node and at each step select randomly one of the nodes and append a
new leaf to it.

We are interested in the analysis of the compaction ratio, relating the tree size
and its minimal DAG size as in \cite{zhang2013redundancy}
for two families of trees that are not simply generated trees. 
The first family consists of recursive trees (Section~\ref{sec:rec_trees}).
The family has been introduced by Moon~\cite{Moon74} and further studied by Meir and Moon in the 70s~\cite{MM78}.
Their motivation was to present a tree model for the spread of epidemics.
The second tree family we are interested in is the class of plane binary increasing tree (Section~\ref{sec:plane}).
It corresponds to the tree model for binary search trees. Both families have been extensively studied in the last
two decades with probabilistic methods~\cite{MS95,drmota03,BDMS08,DIMR09} as well as with
combinatorial ones~\cite{BFS92,KP07,PP07}.

For recursive trees and binary increasing trees, informally speaking we prove that, 
asymptotically, if a tree of size $n$ is compacted, then the resulting structure has on average 
size $\mathcal{O}\left(n/\ln n\right)$, with a lower bound of $\Omega(\sqrt n)$. 

In the context of binary increasing trees the result has already been derived. The upper bound
$\mathcal{O}\left(n/\ln n\right)$ was proved in~\cite{FGM97} as a specific result
in the context of patterns in random binary search trees. The proof is based on some bivariate
generating function analysis in the Analytic Combinatorics context.
The stronger $\Theta$-result has then been proved in~\cite{Devroye98}
based on a preliminary result in~\cite{Fill96}. These papers are based on probability theory
rather than Analytic Combinatorics. But recently other authors~\cite{BL18,BW20} presented new
proofs based on Analytic Combinatorics. We, however, decided to briefly present a further proof
based on Analytic Combinatorics, as it is generic in the following sense: the same approach is valid
for recursive trees as well as for binary increasing trees. 
Especially in order to derive our results, we analyze a perturbation of the differential equation 
defining the tree models, observing that analogous functions related to the increasing labeling
of the tree structure are central in both tree models. And under the
assumption that a certain experimentally supported conjecture is true, almost the same proof can
be used to improve the lower bound and get a $\Theta$-result for both classes.

We thus remark that such a kind of trees are compacted in a more efficient way (in the sense of the number
of remaining nodes) than simply generated trees.
Finally, we end the paper (Section~\ref{sec:data}) with a section dedicated to the compaction of binary search trees (\textsc{bst})
in practice, in order to exhibit the way we can compact the tree structure, but by keeping some extra information
we lose no information (about the labeling of the initial \textsc{bst}). An experimental study is provided by using a
prototype in \emph{python} for our new data structure, the \emph{compacted \textsc{bst}}. The experiments are very
encouraging for the development of such new compacted search tree structures.

So, as a synthesis, Section~\ref{sec:rec_trees} is dedicated to the compaction analysis of recursive trees.
Then Section~\ref{sec:plane} contains the key elements to derive the same result for binary increasing trees
and finally, Section~\ref{sec:data} presents an experimental approach to verify the latter result
in the context of data structures.

\brems 
We note that for all figures we present, we use a postorder traversal of the tree representation in
order to compact them. However, whatever traversal is chosen, the quantitative results are
always identical.

Recall that the size of the compacted tree also equals the number of distinct unlabeled fringe subtrees 
appearing in the original tree.
\erems


%% file: sections/recursive.tex
\section{Recursive trees}
	\label{sec:rec_trees}

The class of recursive trees has been studied by Meir and Moon~\cite{MM78}.
These trees are models in several contexts as e.g. for the study of epidemic spreads,
and thus many quantitative study have focused on this family. Some details are presented either
in~\cite{Drmota09} or in~\cite{flajolet2009analytic}.
Using the classical operators from Analytic Combinatorics, recursive trees can be specified by the so-called boxed product,
or Greene operator,
\[
	\T = \Z \;^{\square} \star \Set (\T),
\]
meaning that the structure of a recursive tree (in the class $\T$) is defined as a root $\Z$
attached to a set of recursive trees (the set may be empty, then $\Z$ is a leaf) and such that the whole structure is
canonically labeled (1, 2, \dots, up to the size). The box in the boxed product indicates that the lowest label goes into
the left component (the atom in this case). The atoms $\Z$ in the structure are therefore labeled increasingly on each path from the root
of the tree to any leaf. See~\cite[Section II.6.3]{flajolet2009analytic} for details
about the constraint labeling operators. The class of recursive trees is also presented in~\cite[Section 1.3]{Drmota09}.

\begin{figure}[h]
\begin{tabular}{c c c}
\resizebox{0.3\textwidth}{!}{
	\begin{tikzpicture}[node distance=15pt]
		\node (f0) {$1$};        
		\node[below of=f0,node distance=25pt] (p) {};   
		\node[right of=p,node distance=40pt] (e) {{$6$}};
		\draw (f0) -- (e);           
		\node[below of=e,node distance=25pt] (z) {$7$};   
		\draw (e) -- (z);  
		\node[below of=z,node distance=25pt] (ppp) {};   
		\node[left of=ppp,node distance=12pt] (zz) {$9$};
		\draw (z) -- (zz);           
		\node[right of=ppp,node distance=12pt] (zzz) {{$11$}};
		\draw (z) -- (zzz);
		\node[below of=zzz,node distance=25pt] (zzzz) {$17$};   
		\draw (zzz) -- (zzzz);   
		
		\node[left of=p,node distance=40pt] (f) {$2$};
		\draw (f0) -- (f);
		\node[below of=f, node distance=25pt] (p2) {};
		\node[right of=p2,node distance=25pt] (pp) {{$4$}};
		\draw (f) -- (pp);   
		\node[below of=pp,node distance=25pt] (ppp) {};   
		\node[left of=ppp,node distance=12pt] (c) {$5$};
		\draw (pp) -- (c);           
		\node[below of=c,node distance=25pt] (cc) {$16$};   
		\draw (c) -- (cc);   

		\node[right of=ppp,node distance=12pt] (e) {{$12$}};
		\draw (pp) -- (e);
		\node[left of=p2, node distance=25pt] (ppp2) {{$3$}};
		\draw (f) -- (ppp2);
		\node[below of=ppp2, node distance=25pt] (pp2) {};
		\node[right of=pp2, node distance=15pt] (pppp2) {{$13$}};
		\draw (ppp2) -- (pppp2);
		\node[left of=pp2, node distance=15pt] (p2) {$8$};
		\draw (ppp2) -- (p2);
		\node[below of=p2, node distance=25pt] (h) {$14$};
		\draw (p2) -- (h);
			
		\node[left of=h,node distance=25pt] (g) {$10$};
		\draw (p2) -- (g);
		\node[right of=h, node distance=25pt] (k) {$15$};
		\draw (p2) -- (k);
	\end{tikzpicture} 
}
&
\resizebox{0.3\textwidth}{!}{
	\begin{tikzpicture}[node distance=15pt, inner sep=0, outer sep=0]
		\node (f0) {};        
		\node[below of=f0,node distance=25pt] (p) {};   
		\node[right of=p,node distance=40pt] (e) {};
		\draw[red] (f0) -- (e);           
		\node[below of=e,node distance=25pt] (z) {};   
		\draw[black] (e) -- (z);  
		\node[below of=z,node distance=25pt] (ppp) {};   
		\node[left of=ppp,node distance=12pt] (zz) {};
		\draw[black] (z) -- (zz);           
		\node[right of=ppp,node distance=12pt] (zzz) {};
		\draw[black] (z) -- (zzz);
		\node[below of=zzz,node distance=25pt] (zzzz) {};   
		\draw[black] (zzz) -- (zzzz);

		\node[left of=p,node distance=40pt] (f) {};
		\draw[red] (f0) -- (f);
		\node[below of=f, node distance=25pt] (p2) {};
		\node[right of=p2,node distance=25pt] (pp) {};
		\draw[red] (f) -- (pp);   
		\node[below of=pp,node distance=25pt] (ppp) {};   
		\node[left of=ppp,node distance=12pt] (c) {};
		\draw[red] (pp) -- (c);           
		\node[below of=c,node distance=25pt] (cc) {};   
		\draw[black] (c) -- (cc); 

		\node[right of=ppp,node distance=12pt] (e) {};
		\draw[black] (pp) -- (e);
		\node[left of=p2, node distance=25pt] (ppp2) {};
		\draw[red] (f) -- (ppp2);
		\node[below of=ppp2, node distance=25pt] (pp2) {};
		\node[right of=pp2, node distance=15pt] (pppp2) {};
		\draw[black] (ppp2) -- (pppp2);
		\node[left of=pp2, node distance=15pt] (p2) {};
		\draw[red] (ppp2) -- (p2);
		\node[below of=p2, node distance=25pt] (h) {};
		\draw[black] (p2) -- (h);
				
		\node[left of=h,node distance=25pt] (g) {};
		\draw[red] (p2) -- (g);
		\node[right of=h, node distance=25pt] (k) {};
		\draw[black] (p2) -- (k);
	\end{tikzpicture}
} 
&
\resizebox{0.26\textwidth}{!}{
	\begin{tikzpicture}[node distance=15pt, inner sep=0, outer sep=0]
		\node (f0) {};        
		\node[below of=f0,node distance=25pt] (p) {};   
		\node[right of=p,node distance=40pt] (e) {};
		\draw[red] (f0) -- (e);           
		
		\node[left of=p,node distance=40pt] (f) {};
		\draw[red] (f0) -- (f);
		\node[below of=f, node distance=25pt] (p2) {};
		\node[right of=p2,node distance=25pt] (pp) {};
		\draw[red] (f) -- (pp);
		\node[below of=pp,node distance=25pt] (ppp) {};   
		\node[left of=ppp,node distance=12pt] (c) {};
		\draw[red] (pp) -- (c);   
		\draw[->,>=latex, black, dotted] (e) to[out=-90, in=30] (pp);  
		
		\node[below of=pp,node distance=25pt] (ppp) {};   
		
		\node[right of=ppp,node distance=12pt] (e) {};
		
		\node[left of=p2, node distance=25pt] (ppp2) {};
		\draw[red] (f) -- (ppp2);
		\node[below of=ppp2, node distance=25pt] (pp2) {};
		\node[right of=pp2, node distance=15pt] (pppp2) {};
		
		\node[left of=pp2, node distance=15pt] (p2) {};
		\draw[red] (ppp2) -- (p2);
		\node[below of=p2, node distance=25pt] (h) {};

		\node[left of=h,node distance=25pt] (g) {};
		\draw[red] (p2) -- (g);
		
		\draw[->,>=latex, black, dotted] (ppp2) to[out=-60, in=30]  (g);
		\draw[->,>=latex, black, dotted] (p2) to[out=-90, in=30] (g);
		\draw[->,>=latex, black, dotted] (c) to[out=-90, in=30] (g);       
		\draw[->,>=latex, black, dotted] (pp) to[out=-60, in=30] (g);
		
		\draw[->,>=latex, black, dotted] (p2) to[out=-60, in=30]  (g);
	\end{tikzpicture}  
}
\end{tabular}
    \caption{\label{fig:small_rec_tree}
		Example of a recursive tree of size $17$}
\end{figure}
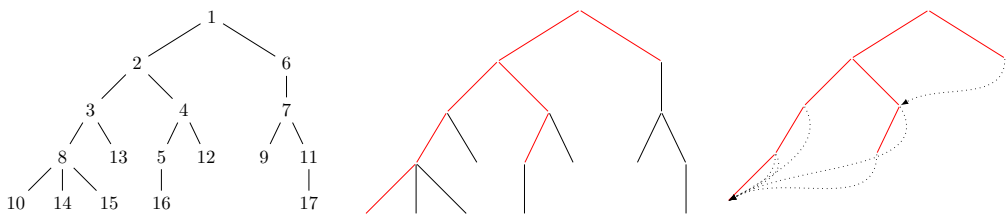
On the left side of Figure~\ref{fig:small_rec_tree} we have represented a recursive tree of size $17$.
The children of a node are put in lexicographic order of their root labels. We remark that the unlabeled structures
underlying the fringe subtrees rooted at $4$ and $7$ have the same unlabeled non-plane structure.
And obviously the leaves are also identical.
So, in the middle of the figure we represent with black edges the fringe subtrees
whose unlabeled non-plane structure has already been seen through a postorder traversal
of the leftmost tree. Finally, on the right side of the figure we replace the multiple occurrences
of a subtree by pointers to the first occurrence. 

\begin{figure}[h]
	\begin{tabular}{c c}
		\includegraphics[width=0.48\textwidth]{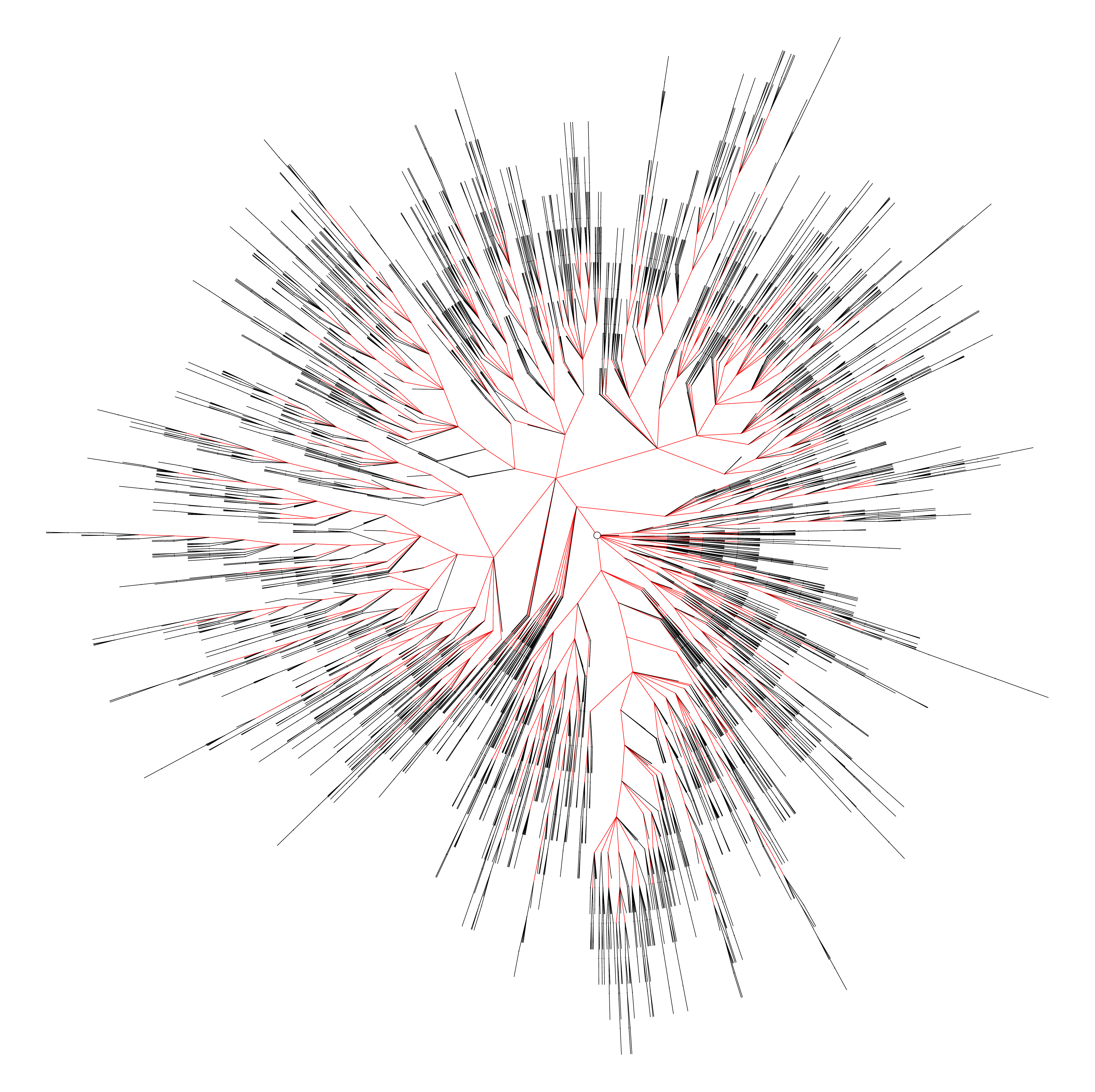}
	&
		\hspace*{-5mm}\includegraphics[width=0.48\textwidth]{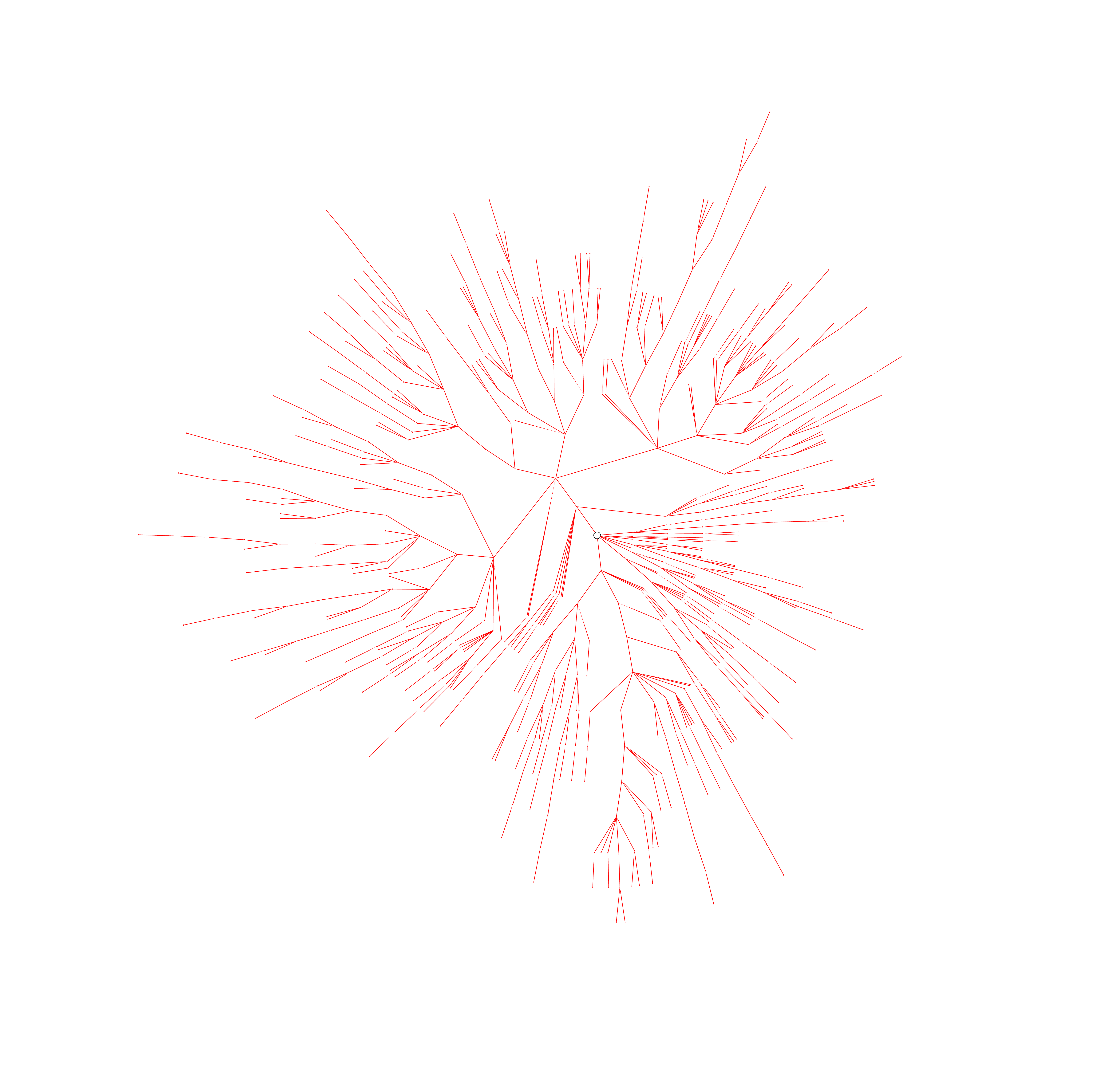}
	\end{tabular}
    \caption{\label{fig:rec_tree_compact}
		(left) A uniformly sampled non-plane recursive tree of size 5,000: Black fringe subtrees are removed
		by the compaction. (right) The red part is of size 663.}
\end{figure}
In Figure~\ref{fig:rec_tree_compact} we have represented a recursive tree structure
containing 5,000 nodes on the left side. It has been uniformly sampled among all trees with the same size.
The original root of the tree is represented using a small circle $\circ$. On the right side
we have depicted the nodes that are kept after the compaction of the latter tree. Only 663 nodes remain.

We define the exponential generating
function $T(z) = \sum_{n\geq 1} T_n \frac{z^n}{n!}$,
where $T_n$ corresponds to the number of trees containing $n$ nodes \emph{i.e.} of \emph{size} $n$.
Using the now classical \emph{symbolic method} from Analytic Combinatorics, from the latter unambiguous
specification we deduce the following functional equation satisfied by $T(z)$:
\[
	T(z) = \int_{0}^{z} \exp(T(v)) \dv.
\]
The unique power series solution satisfying $T(0)=0$ is
\[
	T(z)=\ln \frac{1}{1-z},
\]
whose dominant singularity is $\rho=1$. 
Finally, we get the value $T_n = (n-1)!$.

Let $\mathcal{T}_n$ be the class of recursive trees of size $n$; the size of a tree $\tau$ is defined
as the number of its nodes and is denoted by $|\tau|$. Let $X_n$ be the size of the compacted
tree corresponding to a random recursive tree $\tau$ of size $n$.
In other words, $X_n$ is the number of distinct fringe subtree shapes in $\tau$.
We define $\mathcal{P}$ as the set of P\'olya trees, \textit{i.e.}, non-plane unlabeled trees
such that the degrees of their nodes are arbitrary.
This class of trees is presented in detail in Drmota's book~\cite[Section 1.2.5]{Drmota09}.
and it corresponds to the possible shapes of the recursive trees,
once the increasing labeling has been removed.
We denote by $\mathcal{P}_{\leq n}$ the set of all P\'olya trees with size at most $n$.
Then we have
\begin{equation}\label{equ:sumexp1}
	\mathbb{E} \( X_n \) = \sum_{t \in \mathcal{P}_{\leq n}} \mathbb{P} ( t \ \text{occurs as subtree of} \ \tau )
		= \sum_{t \in \mathcal{P}_{\leq n}} 1 - \mathbb{P} ( t \ \text{does not occur as subtree of} \ \tau ).
\end{equation}
Recall that the tree $t$ corresponds to a tree shape, it is unlabeled,
while $\tau$ is a recursive tree and therefore is increasingly labeled.

\medskip
Now, for a given P\'olya tree $t\in \mathcal{P}$ let us consider a perturbed combinatorial class $\mathcal{S}_t$
that contains all recursive trees except for those that contain a $t$-shape as a (fringe) subtree.
The corresponding exponential generating function satisfies the differential equation
\begin{equation}\label{defors}
	S_t'(z) = \exp(S_t(z)) - P_t'(z),
\end{equation}
where $P_t(z) = \ell(t) \frac{z^{|t|}}{|t|!}$, with $\ell(t)$ denoting the number of ways to
increasingly label the tree shape~$t$.

So, using \eqref{equ:sumexp1} we obtain
\begin{align}
	\mathbb{E} \( X_n \) &= \sum_{t \in \mathcal{P}_{\leq n}}
		\( 1 - \mathbb{P} ( t \ \text{does not occur as shape of a fringe subtree of} \ \tau ) \) \nonumber\\
	&= \sum_{t \in \mathcal{P}_{\leq n}} \( 1 - \frac{[z^n]S_t(z)}{[z^n]T(z)} \). \label{equ:sumexp2}
\end{align}

Therefore, the problem is now essentially reduced to the analysis of the asymptotic behavior of
$[z^n]S_{t}(z)$.

Solving \eqref{defors} we obtain the exponential generating function
\begin{equation}\label{equ:St_solution}
	S_t(z)= \ln \( \frac{1}{1- \int_{0}^{z} \exp(-P_t(v)) \dv} \) - P_t(z).
\end{equation}
Thus, for the dominant singularity $\tilde{\rho}_t$ of $S_t(z)$, the following equation must hold:
\begin{equation}\label{equ:conditionrhotilde}
	\int_{0}^{\tilde{\rho}_t} \exp(-P_t(v)) \dv = 1.
\end{equation}
As $\exp(-P_t(v))<1$ for positive $v$, the dominant singularity $\tilde{\rho}_t$ is greater than 1.
Recall that $\rho$ denotes the dominant singularity of $T(z)$, thus $\rho=1$ and
therefore we write $\tilde{\rho}_t=\rho(1+\epsilon_t)=1+\epsilon_t$ with suitable $\epsilon_t>0$.

\subsubsection*{Notations}
Before we proceed, let us introduce some frequently used notations: For the size and the weight of a P\'olya tree $t$ we use
\[
k:=|t| \qquad\text{ and }\qquad w(t):=\frac{\ell(t)}{|t|!},
\]
respectively. Moreover, let
\[
G(z):=\int_0^z e^{-P_t(v)}\dv=\int_{0}^{z} e^{-w(t)v^{k}} \dv,
\]
if $z\ge 0$ and its complex continuation if $z$ is not a nonnegative real number. With this notation
\eqref{equ:conditionrhotilde} reads as $G(1+\epsilon_t)=1$. By expanding the integrand, we obtain
\[
G(z)=\sum_{\ell\geq 0}(-w(t))^\ell\frac{z^{\ell k+1}}{(\ell k+1)\cdot \ell!},
\]
which shows that $G(z)$ is an entire function.

\subsubsection*{How to proceed}
Taking a random recursive tree of size $n$,
we are interested in the asymptotic behavior of the size of the compacted tree
issued from the compaction of the recursive one.
In order to obtain bounds for this compacted size we proceed as follows:
First, in Lemma~\ref{lem:epsilon}, we compute a upper bound for $\tilde{\rho}_t$.

Then, in Proposition~\ref{lem:asympst}, we provide uniform asymptotics
for the $n$-th coefficient of the generating function $S_t(z)$ when $n$ tends to infinity,
thereby showing that the error term is sufficiently small for what is needed later on. 

The average size of a compacted tree corresponding to a random recursive tree
is expressed as a sum over the forbidden trees. Thereby, the two cases where the
size $k$ of the forbidden tree $t$ is smaller or larger than $\log n$
are treated in a different way: Upper bounds for the size of the compacted tree are derived
in Proposition~\ref{lem:estimate_sum1} (small trees) and
Proposition~\ref{lem:estimate_sum3} (large trees). Finally, Proposition~\ref{thm:recursive_omega},
gives a (crude) lower bound for the size of the compacted tree.

\begin{lem}\label{lem:epsilon}
	Let $S_t(z)$ be the generating function of the perturbed combinatorial class
	(\emph{cf.} Equation~\eqref{defors}) of recursive trees that do not contain a subtree of
	shape $t$ and $\tilde{\rho}_t$ be the dominant singularity of $S_t(z)$ (\emph{cf.}
	Equation~\eqref{equ:conditionrhotilde}). Furthermore, let $k=|t|$ and $w(t)=\ell(t) /k!$ where
	$\ell(t)$ denotes the number of possible increasing labelings of the P\'olya tree $t$. Then
	\[
		\tilde{\rho}_t = 1 + \epsilon_t < 1+\frac{2w(t)}k.
	\]
\end{lem}

\bpf
First observe that the number of increasing labelings of the P\'olya tree $t$ is
 bounded by $(k-1)!$, which gives the very crude bound $w(t) \leq 1/k$ which is valid
 for any tree $t$.

Next, as $\tilde{\rho}_t$ satisfies $G(1+\epsilon_t)=1$, it suffices to show the inequality
$G\(1+\frac{2w(t)}k\)>G(1+\epsilon_t)$. We show the equivalent inequality $G\(1+\frac{2w(t)}k\)-G(1)>G(1+\epsilon_t)-G(1)$:
If $k=2$, then $t$ is a path of length one and therefore $w(t)=1/2$. This gives explicitly
$\int_1^{3/2} e^{-v^2/2}\dv> 1/6$ which is easily verified.\\
If $k\ge 3$, then we have the lower bound
\begin{align*}
	G\(1+\frac{2w(t)}k\)-G(1) &\ge \frac{2w(t)}k \exp\(-w(t)\(1+\frac{2w(t)}k\)^k\) \\
		&\ge \frac{2w(t)}k \exp\(-w(t)\(1+\frac{2}{k^2}\)^k\),
\end{align*}
because $w(t)\le 1/k$.
Then for $k\ge 3$ we have $\(1+\frac2{k^2}\)^k<2$ 
and again, since $w(t)\le 1/k$, we obtain $2e^{-2/3}>1$ and thus
\[
	G\(1+\frac{2w(t)}k\)-G(1) \ge \frac{w(t)}k\cdot 2e^{-2w(t)} > \frac{w(t)}{k+1}.
\]
On the other hand, we have 
\[
	G(1+\epsilon_t)-G(1)= 1-\int_0^1 e^{-w(t) v^k}\dv \le 1- \int_0^1 (1-w(t)v^k)\dv=\frac{w(t)}{k+1},
\]
which implies the assertion. 
\epf

With a similar reasoning as in the above proof a lower bound for $\tilde{\rho}_t$ can be shown: 

\bncor
\label{cor:rho_lower_bound}
With the notations of Lemma~\ref{lem:epsilon} we have the following estimate:
\[
	\tilde{\rho}_t > 1+\frac{w(t)}{k+1}.
\]
\encor

\bncor
\label{cor:asympt}
With the notations of Lemma~\ref{lem:epsilon} we have the following asymptotic relation:
\[
	\tilde{\rho}_t = 1 + \epsilon_t \sim 1+ \frac{w(t)}k, \text{ as } k\to \infty.
\]
\encor

\bpf
Write $G(z)$ as $G(z)=z+R(z)$ with
\begin{equation} \label{R_fctn}
	R(z)=\sum_{\ell\geq 1}(-w(t))^\ell\frac{z^{\ell k+1}}{(\ell k+1)\cdot \ell!}
\end{equation}
As $\tilde{\rho}_t=1+\epsilon_t$ is the smallest positive solution of $G(z)=1$, 
it is the smallest positive zero of $z-1+R(z)$.
From Lemma~\ref{lem:epsilon} we know that $\epsilon_t=\Ord{1/k^2}$ and thus
$\tilde{\rho}_t^k\sim 1$, as $k$ tends to infinity, and
$R(\tilde{\rho}_t)=w(t)\tilde{\rho}_t^{k+1}/(k+1)+\Ord{1/k^3}$. This implies
\[
	\epsilon_t \sim \frac{w(t)}{k+1} \tilde{\rho}_t^{k+1} \sim \frac{w(t)}k,
\]
as desired.
\epf

\brem
In the paper~\cite{GWW16}, which is related to pattern exclusion in recursive trees, 
the same result about the singularity $\tilde{\rho}_t$ is proved.
Using more terms of the expansion of $G(z)$, it is possible to derive a more accurate
asymptotic expression for $\epsilon_t$ (in principle up to arbitrary order).
As an example, we state
\[
	\tilde{\rho}_t = 1+ \frac{w(t)}{k+1} + \frac{w(t)^2 (3k+1)}{(k+1)(4k+2)} +
		\frac{w(t)^3 (29k^3+32k^2+10k+1)}{6(k+1)^2(2k+1)(3k+1)} + \Ord{\frac{w(t)^4}{k}}.
\]
\erem

Note that in the sequel we will have to evaluate the coefficient $[z^n]S_t(z)$ for $n$ tending to
infinity and $|t|$ tending to infinity with $n$ as well. Thus a standard transfer lemma in the
sense of Flajolet and Odlyzko~\cite{flajolet1990singularity} is not sufficient. We need a tight
and uniform error term. In order to find this, we need to know where the second dominant
singularity is, or rather where we can be sure that there will not be any singularity. The next
lemma provides information about an eventually large enough singularity-free region. 

\bl\label{lem:nosingul}
Let $S_t(z)$ be the generating function of the perturbed class of recursive trees
defined in~\eqref{equ:St_solution}. Then $S_t(z)$ has no singularity in the domain 
\[
	\tilde\rho_t<|z|<1+ \frac{\ln (1/w(t))+\ln\ln\ln(1/w(t))}k.
\]
\el

\bpf
Recall that by~\eqref{equ:St_solution} we have
\begin{equation*}
	S_t(z)= \ln \left( \frac{1}{1-G(z)} \right) - P_t(z).
\end{equation*}
Since $G(z)$ is an entire function, the singularities of $S_t(z)$ are exactly the zeros of $G(z)-1$.
Therefore, consider $z_0$ such that $G(z_0)=1$ and write $G(z)=z+R(z)$ with $R(z)$ as in
\eqref{R_fctn}. Then the chosen number $z_0$ must satisfy the inequality 
\begin{align}
	|R(z_0)|\le \frac{|z_0|}{k+1} \sum_{\ell\ge 1}\frac{|w(t)|^\ell |z_0|^{k \ell+1}}{\ell!} 
		< \frac 1k (e^{|w(t)| |z_0|^k} -1).  \label{R_estim}
\end{align}

The first step is to show that $G(z)-1$ does not have any zeros (except $\tilde\rho_t$) 
in a sufficiently large domain, \emph{i.e.} that either $z_0=\tilde\rho_t$ or $|z_0|$ is large. 
We have to approach this in two steps. 

{\bf Case 1:} Assume first that $|z_0|\le 1+\frac{\alpha\ln(1/w(t))}k$ for some $\alpha<1$. 
As the dominant singularity of $S_t(z)$ is $\tilde{\rho}_t$ and $\tilde{\rho}_t>1$,
we must have $|z_0|>1$. Thus, the upper bound on $|z_0|$ implies $|z_0|^k\le
\exp\(\alpha\ln(1/w(t))\)=(1/w(t))^\alpha=o(1/w(t))$ and by \eqref{R_fctn} we obtain then 
\begin{equation} \label{dist_from_1}
1-z_0 = R(z_0)\sim -\frac{w(t)}k z_0^k =o\(\rcp{k}\).
\end{equation}
This implies further that $z_0^k\sim 1$, hence 
$z_0$ is asymptotically equal to a $k$-th root of unity. But then $z_0=\tilde \rho_t$, because
the distance between the other $k$-th roots of unity and 1 
is greater than $1/k$, which contradicts \eqref{dist_from_1}.


{\bf Case 2:} Now, assume that $|z_0|=1+\eta$ with $\alpha \ln(1/w(t))/k < \eta \le
(\ln(1/w(t))+\ln\ln\ln(1/w(t))-\delta)/k$ for some arbitrary but small $\delta>0$. Then 
$w(t) |z_0|^k \le \ln\ln(1/w(t)) e^{-\delta}$ 
and so by \eqref{R_estim} we have then 
\begin{equation} \label{R_est}
	|R(z_0)|\le \frac{\(\ln\rcp{w(t)}\)^{e^{-\delta}}-1}k.
\end{equation}  
But we assumed $|z_0-1|>\alpha\ln(1/w(t))/k$ and so $R(z_0)$ would be too small to
compensate the value of $z_0-1$. Indeed, we observe that in this region
\begin{equation} \label{G_est}
	|G(z)-1|> \rcp{k} +\frac{\alpha\ln\rcp{w(t)}-\(\ln\rcp{w(t)}\)^{e^{-\delta}}}k \ge
		\frac{\gamma\ln\rcp{w(t)}}k
\end{equation} 
holds, where $\gamma$ is a suitable positive constant.

Summarizing what we have so far, we infer that either $z_0=\tilde{\rho}_t$ or 
\[
	|z_0|\ge 1+ \frac{\ln (1/w(t))+\ln\ln\ln(1/w(t))}k,
\]
as claimed. 
\epf

Now we are able to derive a uniform asymptotic expression for the coefficients of $S_t(z)$ with a
sufficiently small error term.
\begin{prop}\label{lem:asympst}
Let $S_t(z)$ be the generating function of the perturbed class of recursive trees
defined in~\eqref{equ:St_solution} and fix a constant $L>2$. Then, uniformly for $D\le |t|\le n$
with $D$ independent of $n$ and sufficiently large, 
the following asymptotic relations hold, depending of the magnitude of $w(t)$:  
\begin{itemize}
\item If $\ln\rcp{w(t)}=o\(\sqrt k\,\)$, then 
the coefficients of $S_t(z)$ behave asymptotically as follows:  
	\[
		[z^n]S_t(z) = \frac{\tilde{\rho}_t^{-n}}n 
		\(1+\Ord{\frac1{\sqrt k} \(\frac{kcw(t)}{\ln\ln \rcp{w(t)}}\)^{n/k}}\),
		\text{ as } \nti,
	\]
where $c$ is an arbitrary constant satisfying $c>1$. 
\item If $\ln\rcp{w(t)}=\Omega\(\sqrt k\,\)$ and $\ln\rcp{w(t)}\le L k$, then 
	\[
		[z^n]S_t(z) = \frac{\tilde{\rho}_t^{-n}}n
                \(1+\Ord{\exp\(\frac nk\cdot \frac{\ln(L+1)}L 
		\ln(kw(t))\)}\), \text{ as } \nti.
	\]
\item If $\ln\rcp{w(t)}>Lk$, then 
	\[
		[z^n]S_t(z) = \frac{\tilde{\rho}_t^{-n}}n
                \(1+\Ord{ \ln (k) \exp\(-n \(\ln\(\ln\rcp{w(t)}-\ln k\)-\ln k\)\)}\), 
		\text{ as } \nti.
	\]
\end{itemize}
\end{prop}

\begin{proof}
Notice that $G'(\tilde{\rho}_t)=\exp\(-w(t)\tilde{\rho}_t^k\)\neq 0$ and therefore
$\tilde{\rho}_t$ is a simple zero of $G(z)-1$.
Thus $G(z)-1=(z-\tilde{\rho}_t)\tilde G(z)$ where $\tilde G(z)$ is analytic in the considered 
domain and does not have any zeros there.
Thus,
\[
	S_t(z)= \ln \left( \frac{1}{1-G(z)} \right) - P_t(z)
		=-\ln\(1-\frac{z}{\tilde{\rho}_t}\)-\ln(\tilde{\rho}_t \tilde G(z)) -P_t(z),
\]
where, apart from the first summand, there are no singularities in 
$|z|< 1+ \frac{\ln(1/w(t))+\ln\ln\ln(1/w(t))}k$ (see Lemma~\ref{lem:nosingul}). 
Expanding the logarithm gives
\[
[z^n]S_t(z)=\frac{\tilde{\rho}_t^{-n}}n\(1+\Ord{n\tilde{\rho}_t^n[z^n]\ln\tilde G(z)}\)
\]
and we want to estimate $[z^n]\ln\tilde G(z)$ using Cauchy's estimate. 
%
%
%
Therefore we use the integration contour 
\[
\Gamma:=\left\{z:\; |z|=1+\frac{\ln \rcp{w(t)}+\ln\ln\ln\rcp{w(t)}-\delta}k\right\}
\] 
for some small $\delta>0$, which we split into a part $\Gamma_1$
where $|z-1|\le 5\ln (1/w(t))/k$ and its complement $\Gamma_2$.

As we want to estimate the logarithm of $\tilde G(z)= (G(z)-1)/(z-\tilde{\rho}_t)$, we need an upper
and a lower bound for $\tilde G(z)$.

First of all, note that on the whole integration contour certain useful inequalities hold,
provided that $k$ is sufficiently large:
\begin{align*}
|z-\tilde{\rho}_t|&\ge |z-1|-|1-\tilde{\rho}_t|\ge |z-1|-\frac{2w(t)}{k}\ge 
|z-1|\left(1-\frac{2w(t)}{\ln \rcp{w(t)}}\right) \ge \frac{|z-1|}2,  \\
|z-\tilde{\rho}_t|&\le |z-1|+|\tilde{\rho}_t-1|\le |z-1|+\frac{2w(t)}{k}\le 
|z-1|\left(1+\frac{2w(t)}{\ln \rcp{w(t)}}\right) \le 2|z-1|,
\end{align*}
which is true, because $|1-\tilde{\rho}_t|<2w(t)/k$ due to Lemma~\ref{lem:epsilon} and
$\ln(1/w(t))/k\le |z-1|$. For $z\in \Gamma_1$ the upper bound can be slightly improved:
Indeed, we even have $|z-\tilde{\rho}_t|\le |z-1|$. Moreover, recall the inequality
\[
	|R(z)|\le \frac{\alpha \ln\rcp{w(t)} - 1}k,      
\]
which follows from \eqref{R_est}. On $\Gamma_1$ we also have 
\begin{equation} \label{z_rho}
	\frac{\ln (1/w(t))}{2k}\le |z-\tilde{\rho}_t|\le \frac{5\ln (1/w(t))}{k}.
\end{equation}

From all these inequalities we infer a universal upper bound 
(for all $z\in \Gamma_1\cup \Gamma_2$) for $\tilde G(z)$:
\[
	\left|\frac{G(z)-1}{z-\tilde{\rho}_t}\right|\le
		\frac{|z-1|}{|z-\tilde{\rho}_t|}+\frac{|R(z)|}{|z-\tilde{\rho}_t|}\le 2+\frac{2(e-1)}{\ln (1/w(t))}\le 3.
\]
Here we used that the first inequality in \eqref{z_rho} actually holds
on the whole integration contour. Using \eqref{G_est} and the second inequality in \eqref{z_rho}
we get for $z\in\Gamma_1$ the lower bound
\[
	\left|\frac{G(z)-1}{z-\tilde{\rho}_t}\right|\ge \frac\gamma5 >\rcp5.
\]
These two bounds and the fact that the length of the curve $\Gamma_1$ is less than
$10\ln(1/w(t))/k$ imply
\begin{equation*}
	\left|[z^n]\ln\tilde G(z)\right|\le \(1+\frac{\ln \rcp{w(t)}+\ln\ln\ln\rcp{w(t)}-\delta}k\)^{-n}
		\frac{10\ln \(\rcp{w(t)}\) \ln 5}{k} + \frac1{2\pi}\int_{\Gamma_2} \frac{|\ln\tilde G(z)|}{|z|^{n+1}} |\dz|.
\end{equation*}

Turning to $\Gamma_2$, we obtain the lower bound
\[
	\left|\frac{G(z)-1}{z-\tilde{\rho}_t}\right|\ge
		\frac{|z-1|}{|z-\tilde{\rho}_t|} - \frac{|R(z)|}{|z-\tilde{\rho}_t|} \ge \frac12 -
		\frac{\alpha\ln\rcp{w(t)}-1}k
		\frac{k}{5\ln \rcp{w(t)}} \ge \frac1{10},
\]
and so $|\ln\tilde G(z)|$ is bounded on $\Gamma_2$.

Finally, let $M:=\max(\ln(10),10\ln(1/w(t))\ln(10)/k)$. Altogether the above estimates show 
that for sufficiently large $k$ we have
\begin{align}
	n\tilde{\rho}_t^n |[z^n]\ln\tilde G(z)|&\le
			n\tilde{\rho}_t^n \(1+\frac{\ln
\rcp{w(t)}+\ln\ln\ln\rcp{w(t)}-\delta}k\)^{-n} \(\frac{10\ln\(\rcp{w(t)}\)\ln 5}k+\ln(10)M\)
\nonumber \\
& \le n \(1+\frac{\ln 
\rcp{w(t)}+\ln\ln\ln\rcp{w(t)}-2\delta}k\)^{-n} \(\frac{10\ln\(\rcp{w(t)}\)\ln 5}k+\ln(10)M\) 
\nonumber \\
&=\Ord{n\(1+\frac{\ln \frac1{w(t)}-\ln k+\ln\ln\ln\rcp{w(t)}-2\delta}k+\frac{\ln
n}n\)^{-k\cdot\frac nk}
\cdot\frac{\ln\rcp{w(t)}}k} \label{last} \\
&=\Ord{\frac{\ln\rcp{w(t)}}k \(\frac{w(t)ke^{2\delta}}{\ln\ln \rcp{w(t)}}\)^{n/k}},
\nonumber 
\end{align}
where the last step is only true in the case where $\ln(1/w(t))=o\(\sqrt k\,\)$ 
and yields the desired result after all.

In all the other cases, only the last step is different.
Indeed, going back to \eqref{last}, we can estimate $\ln\ln\ln\rcp{w(t)}-2\delta>0$ and thus 
\[ 
	n\tilde{\rho}_t^n |[z^n]\ln\tilde G(z)|=\Ord{\frac{\ln\rcp{w(t)}}k (1+X)^{-n}},
\] 
with $X=(\ln(1/w(t))-\ln k)/k$. 

If $\ln(1/w(t))=\Omega\(\sqrt k\,\)$, but $\ln(1/w(t))\le L k$, we write $(1+X)^{-n}=\exp(-n\ln
(1+X))$ and get the final result by using $\ln(1+X)\ge X\ln(L+1)/L$, which is true for $0\le X\le
L$. The prefactor $\ln(1/w(t))/k$ is bounded by $L$ in the considered case. 

And finally, if $\ln(1/w(t))> L k$ (and so $X>L$), then simply use $\ln(1+X)>\ln X$. This
yields 
\[ 
(1+X)^{-n}\le \exp\(-n \(\ln\(\ln\rcp{w(t)}-\ln k\)-\ln k\)\). 
\]
As $w(t)\le (k-1)!$, we get $\ln(1/w(t))/k=\Ord{\ln k}$ and the proof is complete. 
\end{proof}

The uniform error term in Proposition~\ref{lem:asympst} allows us to derive a simple upper bound
for $k$ not too large. It turns out that the bound in Corollary~\ref{cor:rho_lower_bound} is
actually good enough to cover the error term from Proposition~\ref{lem:asympst}. 

\bncor \label{cor:coeff_lower_bound}
If $k$ is sufficiently large, then 
\[
[z^n]S_t(z) \le \frac1n \(1+\frac{w(t)}{k+1}\)^{-n},
\] 
as $n$ tends to infinity and $k=\Ord{\sqrt n}$.  
\encor

\bpf
We know from Proposition~\ref{lem:asympst} that $[z^n]S_t(z) = \tilde{\rho}_t^{-n} n^{-1} (1+r_n)$
with $r_n=o(1)$. Thus we must show that 
\[
r_n\le \tilde{\rho}_t^{n} \(1+\frac{w(t)}{k+1}\)^{-n}-1 = \(1+ \Ord{w(t)^2}{k}\)^n -1. 
\]
As $r_n$ tends to 0, the inequality is trivial if $nw(t)^2/k$ does not tend to 0, as in this case
the right-hand side grows exponentially. Otherwise we are left with having to show the estimate
$r_n=\Ord{nw(t)^2/k}$. Let us compare $nw(t)^2/k$ with the exponential part of the error term
given by Proposition~\ref{lem:asympst}. In the case where $w(t)$ is large ($\ln\rcp{w(t)}=o(\sqrt
k\,)$) this gives
\begin{align*} 
&\(\frac{kcw(t)}{\ln\ln \rcp{w(t)}}\)^{n/k} \frac k{nw(t)^2} \\ 
&\qquad = \exp\( \(-\frac nk +2\)\ln\rcp{w(t)}
+\(\frac nk +1\)\ln k -\frac nk \ln\ln\ln\rcp{w(t)} +\frac nk \ln c -\ln n\) \\
& \qquad \le \exp\(3\ln k -\frac nk \ln\ln\ln k +\frac nk \ln c -\ln n\), 
\end{align*} 
where the inequality holds because of $\ln\rcp{w(t)}\ge \ln k$. As our assumptions imply
$n/k\to\infty$ and so the dominant term in the exponent, $-\frac nk \ln\ln\ln k$, is negative, 
we obtain $r_n=o(nw(t)^2/k)$ as desired. 

In the case where $w(t)$ has intermediate size, the difference of the logarithms of the
exponential term in the error and of $nw(t)^2/k$ is equal to 
\[
\(-\frac nk \frac{\ln(L+1)}{L}+2\) \ln\rcp{w(t)}+ \(\frac nk \frac{\ln(L+1)}{L}+1\)\ln k -\ln n
\]
which is negative if $k=\Ord{\sqrt n}$.

Finally, if $w(t)$ is small, then the difference of the logarithms equals
\begin{align*}  
&-n\(\ln\(\ln\rcp{w(t)}-\ln k\)-\ln k\)+2\ln\rcp{w(t)} +\ln k-\ln n 
\\ 
&\qquad\le -n(\ln((L-1)k)-\ln
k)+2\ln\rcp{w(t)}+\ln k -\ln n \\
&\qquad\le -n\ln(L-1)+2\ln\rcp{w(t)}+\ln k -\ln n 
\end{align*} 
which is again negative if $k=\Ord{\sqrt n}$.
\epf

Within this section many logarithms that occur are with respect to the base
$\frac{1}{\sigma}\approx 2.9955765$,
where $\sigma \approx 0.3383218$ denotes the dominant singularity of the generating function of
P\'olya trees (\emph{cf.} \cite[Section~VII.5]{flajolet2009analytic}).
We thus use the notation $\log_{\frac1\sigma}$ for the logarithm with respect to base
$\frac1\sigma$.

Now we decompose the sum \eqref{equ:sumexp2} into
\begin{equation}\label{equ:threesums}
	\mathbb{E} \left( X_n \right) = \sum_{\substack{t \in \mathcal{P}_{\leq n}\\ k < \log_{\frac1\sigma} n}}
		 \left( 1 - \frac{[z^n]S_t(z)}{[z^n]T(z)} \right) +
		 \sum_{\substack{t \in \mathcal{P}_{\leq n}\\ k \geq  \log_{\frac1\sigma} n}} \left( 1 - \frac{[z^n]S_t(z)}{[z^n]T(z)} \right),
\end{equation}
and investigate the two sums individually, starting with the first one, whose summands are
probabilities and thus bounded by 1.
\begin{prop}\label{lem:estimate_sum1}
	The first sum in (\ref{equ:threesums}) behaves asymptotically as
	\[ 
	\sum_{\substack{t \in \mathcal{P}_{\leq n}\\ k < \log_{\frac1\sigma} n}} \( 1 - \frac{[z^n]S_t(z)}{[z^n]T(z)} \)
	\underset{n\rightarrow \infty}= \mathcal{O} \( \frac{n}{\sqrt{(\ln n)^3}} \).
	\]
\end{prop}
\begin{proof}
Remember that we have set $k:=|t|$. Furthermore, we denote by $P(z)$ the generating function of
P\'olya trees and by $\sigma$ its dominant singularity. Then
\begin{align*} 
\sum_{\substack{t \in \mathcal{P}_{\leq n}\\ k< \log_{\frac1\sigma} n}} 
\left( 1 - \frac{[z^n]S_t(z)}{[z^n]T(z)} \right)
& \leq \sum_{\substack{t \in \mathcal{P}_{\leq n}\\ k < \log_{\frac1\sigma} n}} 1 =
\sum_{k < \log_{\frac1\sigma} n} [z^k]P(z) 
\\
& \sim \frac{1}{1-\sigma} [z^{\lfloor \log_{\frac1\sigma} n \rfloor}]P(z) =
		\mathcal{O} \( \frac{\sigma^{- \lfloor \log_{\frac1\sigma} n \rfloor}}{\sqrt{(\log_{\frac1\sigma} n)^3}} \).
\end{align*} 
Since $\log_{\frac1\sigma} n$ has the base $1 / \sigma$, we estimate $\sigma^{- \lfloor
\log_{\frac1\sigma} n \rfloor} \leq n$,
which completes the proof.
\end{proof}

In order to analyze the second sum from \eqref{equ:threesums} we rely on counting arguments,
which were presented in \cite[Remark~4.2]{GWW16}. For the sake of self-containedness we restate the
counting arguments here. 
\begin{prop}\label{lem:estimate_sum3}
The second sum in \eqref{equ:threesums} behaves asymptotically as
	\[ 
		\sum_{\substack{t \in \mathcal{P}_{\leq n}\\ k \geq  \log_{\frac{1}{\sigma} n}}}
		\left( 1 - \frac{[z^n]S_t(z)}{[z^n]T(z)} \right)= 
		\mathcal{O} \( \frac{n}{ \log_{\frac1\sigma} n} \).
	\]
\end{prop}
\begin{proof}
Remember that we have set $k:=|t|$ and $k$ tends to infinity in this proof. We are interested in
$1 - \frac{[z^n]S_t(z)}{[z^n]T(z)}$, the
probability that a tree of size $n$ contains a fringe subtree of shape $t$. 
	
We start with a counting argument, allowing multiple counting, to construct a tree of size $n$
having a fringe subtree of shape $t$. Let $\nu$ denote the root label of $t$ in the tree of size
$n$. If several occurrences of $t$ do appear, we consider one of them. 

First suppose $k<n$. Then choose a tree of size $n-k$ to which $t$ will be attached. Recall that
the number of possible choices for that tree equals $(n-k-1)!$.  The number of ways to choose
the labels of $t$ is $\binom{n-\nu}{k-1}$, as $\nu$ is the smallest label in $t$ and $|t|=k$.
Once the labels for $t$ have been chosen, there are $\ell(t)$ possibilities to distribute them
over the vertices of $t$ in order to obtain a proper labeling. The initially chosen (and already
labeled) tree of size $n-k$ gets the remaining labels (that have not been chosen for $t$), which
replace the original label in an order-preserving way. Finally, there are $\nu-1$
possible parent nodes to which $t$ can be attached.

Putting all this together, we get the number of all recursive trees of size $n$ having $t$ as a
fringe subtree, but each counted as many times as there are occurrences of $t$. This is clearly an
upper bound. We obtain 
\begin{align}
	1 - \frac{[z^n]S_t(z)}{[z^n]T(z)} &
		\leq  \frac{(n-k-1)!}{(n-1)!} \sum_{\nu = 2}^{n-k+1} (\nu-1)\binom{n-\nu}{k-1} 
		\ell(t) \label{eq:multi} \\
	&= \frac{\ell(t)}{(k-1)!} \frac{n}{k(k+1)} = \frac{nw(t)}{k(k+1)}. \nonumber 
\end{align}

Now let $k=n$. This means that we are interested in the probability that a recursive tree has
shape $t$. In this case,
\[
	1 - \frac{[z^n]S_t(z)}{[z^n]T(z)} = \frac{\ell(t)}{(n-1)!} = nw(t).
\]	

Now we apply this to the sum we want to estimate. Recall that $\sum_{t\in\mathcal P_k} w(t)= 1/k$.
We get 
%
%
\begin{align*}
	\sum_{\substack{t \in \mathcal{P}_{\leq n}\\ k \geq \log_{\frac1\sigma} n}} 
	\left( 1 - \frac{[z^n]S_t(z)}{[z^n]T(z)} \right)
	& \leq n \sum_{t \in \mathcal{P}_{n}} w(t) 
		+\sum_{k \geq \log_{\frac1\sigma} n} \frac{n}{k+1} \sum_{t\in\mathcal{P}_k} w(t) \\
	& = 1+ \sum_{k \geq  \log_{\frac1\sigma} n} \frac n{k(k+1)} \\
	& = 1+ \sum_{k \geq  \log_{\frac1\sigma} n} n \left(\frac1k - \frac1{k+1}\right) 
	= \Theta\( \frac{n}{ \log_{\frac1\sigma} n} \) \qedhere
\end{align*} 
\end{proof}

\begin{thm}\label{thm:bounded_expectation}
Let $X_n$ be the size of the compacted tree corresponding to a
random recursive tree $\tau$ of size $n$. Then
\[
	\mathbb{E} \left( X_n \right) \underset{n\rightarrow \infty}= 
	\mathcal{O} \left( \frac{n}{\ln n} \right).
\]
\end{thm}
\begin{proof}
The result follows directly by combining the previous propositions.
\end{proof}

Finally, we prove a lower bound for the average size of the compacted tree
based on a random recursive tree of size $n$.

\begin{prop}\label{thm:recursive_omega}
Let $\mathcal{P}_{\leq n}$ denote the class of P\'olya trees of size at most $n$. Then
\[
	\sum_{\substack{t \in \mathcal{P}_{\leq n}\\ \log_{\frac1\sigma} n \le k \le \sqrt
	n}} \left( 1 - \frac{[z^n]S_t(z)}{[z^n]T(z)} \right)
	\underset{n\rightarrow \infty}=	\Omega \left( \sqrt{n} \right).
\]
\end{prop}

\begin{proof}
For the sake of simplified reading we will use the abbreviation $\sum_{t} := \sum_{t \in
\mathcal{P}_k}$ in this proof.

First, we use Corollary~\ref{cor:coeff_lower_bound} and the inequality $(1+x)^{-n} \leq \exp\(-n x
+ \frac{n x^2}{2}\)$ in order to estimate
\begin{align}
	A_n:=\sum_{k = \lfloor\log_{\frac1\sigma} n\rfloor}^{\lfloor \sqrt n \rfloor} 
	\sum_t \left(1-\frac{[z^n]S_t(z)}{[z^n]T(z)}\right) 
	&\ge 
	\sum_{k = \lfloor\log_{\frac1\sigma} n\rfloor}^{\lfloor \sqrt n \rfloor} 
	\sum_t \left( 1 - \(1+ \frac{w(t)}{k+1}\)^{-n}\right)
	\nonumber
	\\
	&\geq 
	\sum_{k = \lfloor\log_{\frac1\sigma} n\rfloor}^{\lfloor \sqrt n \rfloor}
	\sum_t \left(1- \exp\(-\frac{nw(t)}{k+1} +   
	\frac{nw(t)^2}{(k+1)^2}\)\right).   \label{equ:est_for_lower_bound}
\end{align} 
Since $x\mapsto 1- \exp\(-nx + \frac{n x^2}{2}\)$, $0 \le x\le 2$, is a concave nonnegative
function with a zero in the origin and $x=w(t)/(k+1)$ certainly falls in this range for all $t$,
we can estimate the inner sum in \eqref{equ:est_for_lower_bound}, which yields
\begin{align*} 
	A_n 
	&\ge 
	\sum_{k = \lfloor\log_{\frac1\sigma} n\rfloor}^{\lfloor \sqrt n \rfloor}
	\left(1- \exp\(-n \sum_t \frac{w(t)}{k+1} 
	+ n \(\sum_t \frac{w(t)}{k+1}\)^2\)\)	
\end{align*} 
As $\sum_t w(t)\le 1/k$, we get
\begin{align*}
	A_n
	&\geq
	\sum_{k = \lfloor\log_{\frac1\sigma} n\rfloor}^{\lfloor \sqrt n \rfloor}
	\(1- \exp\(-\frac{n}{(k+1)^2} + \Ord{\frac{n}{k^4}} \)\) \\
	& \underset{n\rightarrow\infty}\sim \int_{\log_{\frac1\sigma} n}^{\sqrt n} \(1-
	\exp\(-\frac{n}{x^2} + \Ord{\frac{n}{x^4}}\)\) \dx \\
	& = \sqrt{n} \int_{n^{-1/2} \log_{\frac1\sigma} n}^{1} \(1- \exp\(-\frac{1}{y^2} +
	\Ord{\frac{1}{ny^4}}\)\) \dy.
\end{align*}
Since the integral is convergent this gives a lower bound that is $\Theta(\sqrt{n})$.
\end{proof}

We strongly believe that the upper bound presented in Theorem~\ref{thm:bounded_expectation} is in
fact the actual order of magnitude. Unfortunately, we cannot prove this. It seems that a finer
knowledge on the distribution of the values of $w(t)$ is necessary. 

\begin{conj}\label{c1}
If $k\ge \log_{\rcp\sigma} n$, then $\sum_{t\in \mathcal P_k} w(t)^2=\Ord{1/n}$. 
\end{conj}

It is not easy to carry out experiments to support or disprove this conjecture. But for small
value of $n$ this works and they seem to confirm the conjecture. If it is true, then our
conjecture on the order of magnitude $\E(X_n)$ is true as well. 

\begin{thm}\label{thetabound}
If Conjecture~\ref{c1} is true, then 
\[
	\sum_{\substack{t \in \mathcal{P}_{\leq n}\\ \log_{\frac1\sigma} n \le k \le \sqrt
        n}} \left( 1 - \frac{[z^n]S_t(z)}{[z^n]T(z)} \right)
        \underset{n\rightarrow \infty}= \Omega \left( \frac n{\ln n} \right).
\]
Consequently, then $\E(X_n) = \Theta(n/\ln n)$. 
\end{thm}

\bpf
Let us again use the notation $\sum_{t} := \sum_{t \in \mathcal{P}_k}$. 
Then by Corollary~\ref{cor:coeff_lower_bound} we have
\begin{align*} 
	\sum_{k = \lfloor\log_{\frac1\sigma} n\rfloor}^{\lfloor \sqrt n \rfloor}
        A_n:=\sum_t \left(1-\frac{[z^n]S_t(z)}{[z^n]T(z)}\right)
        &\ge
        \sum_{k = \lfloor\log_{\frac1\sigma} n\rfloor}^{\lfloor \sqrt n \rfloor}
        \sum_t \left( 1 - \(1+ \frac{w(t)}{k+1}\)^{-n}\right).
\end{align*} 
The function $f(x)=1-(1+x)^{-n}$ is concave, monotonically increasing for $x\ge0$ and nonnegative
there. Moreover, $f(0)=0$. Thus $f(x)\ge xf'(x)$, since the slope at some $x_0>0$ is flatter than
the slope at $0$ and so the line $x\mapsto xf'(x_0)$ stays below the graph of $f$ at least until
$x=x_0$. This implies 
\[
	A_n\ge \sum_{k = \lfloor\log_{\frac1\sigma} n\rfloor}^{\lfloor \sqrt n \rfloor}
		\sum_t \frac{nw(t)}{k+1} \left( 1 + \frac{w(t)}{k+1}\)^{-n-1} 
	= \sum_{k = \lfloor\log_{\frac1\sigma} n\rfloor}^{\lfloor \sqrt n \rfloor} \frac{n}{k(k+1)} 
		\sum_t kw(t) \(1+\frac{w(t)}{k+1}\)^{-n-1}.
\]
Now observe that $\sum_t kw(t)=1$ and that $g(x)=1/(1+x)^{n+1}$ is a convex function. Thus the
last sum is a convex linear combination of values of $g(x)$ and so Jensen's inequality gives 
\[
	A_n\ge \sum_{k = \lfloor\log_{\frac1\sigma} n\rfloor}^{\lfloor \sqrt n \rfloor} \frac{n}{k(k+1)}
		\(1+\sum_t\frac k{k+1} w(t)^2\)^{-n-1}. 
\]
Under our assumption that Conjecture~\ref{c1} is true, this can be further transformed into 
\[
A_n\ge \sum_{k = \lfloor\log_{\frac1\sigma} n\rfloor}^{\lfloor \sqrt n \rfloor} \frac{n}{k(k+1)}
\(1+\Ord{\rcp n}\)^{-n-1} = \Theta\(\frac n{\ln n}\). \qedhere
\]
\epf

%% file: sections/plane.tex
\section{Plane increasing binary trees}
	\label{sec:plane}

As already mentioned in the introduction, the main result of this section related to the size of
the compaction of a random binary increasing tree (or a random binary search tree) has already
been proved. But here we want to show that the methodology of the previous section is applicable
to other classes of increasing trees as well. Thus we aim at presenting a new proof of this known
result based on the same approach as the one we used for random
recursive trees. Thus, many proofs will only be sketched. 

\medskip
Plane binary increasing trees have a classical specification in the context
of Analytic Combinatorics, once again by using the Greene operator, or boxed product,
allowing to define increasing labeling constraint for decomposable objects.
Thus the specification of this class $\mathcal{T}$ is
\begin{equation}
\label{spec:bin_tree}
	\mathcal{T} = \Z \;^{\square} \star \( 1 +\mathcal{T} \)^2.
\end{equation}
This specification defines a tree to be rooted with an atom $\Z$
associated to a pair of elements that are either the empty element (representing no subtree)
or a subtree itself from the class $\mathcal{T}$.
Once again the operator $\cdot \;^{\square} \star \cdot$ ensures the fact that the smallest
available label must be used for the atom $\Z$.

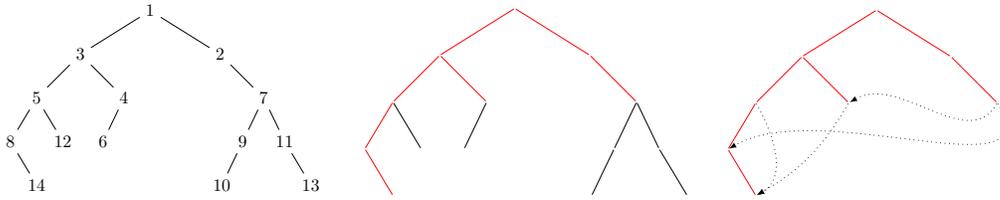
\begin{figure}[h]
\begin{tabular}{c c c}
\resizebox{0.3\textwidth}{!}{
	\begin{tikzpicture}[node distance=15pt]
		\node (f0) {$1$};        
		\node[below of=f0,node distance=25pt] (p) {};   
		\node[right of=p,node distance=40pt] (e) {{$2$}};
		\draw (f0) -- (e);           
		\node[below of=e,node distance=25pt] (z0) {};
		\node[right of=z0,node distance=25pt] (z) {$7$};      
		\draw (e) -- (z);  
		\node[below of=z,node distance=25pt] (ppp) {};   
		\node[left of=ppp,node distance=12pt] (zz) {$9$};
		\draw (z) -- (zz);           
		\node[below of=zz,node distance=25pt] (zzz) {};   
		\node[left of=zzz,node distance=12pt] (zzzz) {$10$};
		\draw (zz) -- (zzzz);           

		\node[right of=ppp,node distance=12pt] (zzz) {$11$};
		\draw (z) -- (zzz);
		\node[below of=zzz, node distance=25pt] (h) {};
		\node[right of=h, node distance=15pt] (k) {$13$};
		\draw (zzz) -- (k);

		\node[left of=p,node distance=40pt] (f) {$3$};
		\draw (f0) -- (f);
		\node[below of=f, node distance=25pt] (p2) {};
		\node[right of=p2,node distance=25pt] (pp) {{$4$}};
		\draw (f) -- (pp);   
		\node[below of=pp,node distance=25pt] (ppp) {};   
		\node[left of=ppp,node distance=12pt] (c) {$6$};
		\draw (pp) -- (c);           
		
		\node[left of=p2, node distance=25pt] (ppp2) {$5$};
		\draw (f) -- (ppp2);
		\node[below of=ppp2, node distance=25pt] (pp2) {};
		\node[right of=pp2, node distance=15pt] (pppp2) {$12$};
		\draw (ppp2) -- (pppp2);
		\node[left of=pp2, node distance=15pt] (p2) {$8$};
		\draw (ppp2) -- (p2);
		\node[below of=p2, node distance=25pt] (h) {};
		\node[right of=h, node distance=15pt] (k) {$14$};
		\draw (p2) -- (k);
	\end{tikzpicture} 
}
&
\resizebox{0.3\textwidth}{!}{
	\begin{tikzpicture}[node distance=15pt, inner sep=0, outer sep=0]
		\node (f0) {};        
		\node[below of=f0,node distance=25pt] (p) {};   
		\node[right of=p,node distance=40pt] (e) {};
		\draw[red] (f0) -- (e);           
		\node[below of=e,node distance=25pt] (z0) {};
		\node[right of=z0,node distance=25pt] (z) {};      
		\draw[red] (e) -- (z);  
		\node[below of=z,node distance=25pt] (ppp) {};   
		\node[left of=ppp,node distance=12pt] (zz) {};
		\draw[black] (z) -- (zz);           
		\node[below of=zz,node distance=25pt] (zzz) {};   
		\node[left of=zzz,node distance=12pt] (zzzz) {};
		\draw[black] (zz) -- (zzzz);           

		\node[right of=ppp,node distance=12pt] (zzz) {};
		\draw[black] (z) -- (zzz);
		\node[below of=zzz, node distance=25pt] (h) {};
		\node[right of=h, node distance=15pt] (k) {};
		\draw[black] (zzz) -- (k);

		\node[left of=p,node distance=40pt] (f) {};
		\draw[red] (f0) -- (f);
		\node[below of=f, node distance=25pt] (p2) {};
		\node[right of=p2,node distance=25pt] (pp) {};
		\draw[red] (f) -- (pp);   
		\node[below of=pp,node distance=25pt] (ppp) {};   
		\node[left of=ppp,node distance=12pt] (c) {};
		\draw[black] (pp) -- (c);           
		
		\node[left of=p2, node distance=25pt] (ppp2) {};
		\draw[red] (f) -- (ppp2);
		\node[below of=ppp2, node distance=25pt] (pp2) {};
		\node[right of=pp2, node distance=15pt] (pppp2) {};
		\draw[black] (ppp2) -- (pppp2);
		\node[left of=pp2, node distance=15pt] (p2) {};
		\draw[red] (ppp2) -- (p2);
		\node[below of=p2, node distance=25pt] (h) {};
		\node[right of=h, node distance=15pt] (k) {};
		\draw[red] (p2) -- (k);
	\end{tikzpicture} 
} 
&
\resizebox{0.3\textwidth}{!}{
	\begin{tikzpicture}[node distance=15pt, inner sep=0, outer sep=0]
		\node (f0) {};        
		\node[below of=f0,node distance=25pt] (p) {};   
		\node[right of=p,node distance=40pt] (e) {};
		\draw[red] (f0) -- (e);           
		\node[below of=e,node distance=25pt] (z0) {};
		\node[right of=z0,node distance=25pt] (z) {};      
		\draw[red] (e) -- (z);  
		
		\node[left of=p,node distance=40pt] (f) {};
		\draw[red] (f0) -- (f);
		\node[below of=f, node distance=25pt] (p2) {};
		\node[right of=p2,node distance=25pt] (pp) {};
		\draw[red] (f) -- (pp);   
		\node[below of=pp,node distance=25pt] (ppp) {};   
		
		\node[left of=p2, node distance=25pt] (ppp2) {};
		\draw[red] (f) -- (ppp2);
		\node[below of=ppp2, node distance=25pt] (pp2) {};
	
	    \node[left of=pp2, node distance=15pt] (p2) {};
		\draw[red] (ppp2) -- (p2);
		\node[below of=p2, node distance=25pt] (h) {};
		\node[right of=h, node distance=15pt] (k) {};
		\draw[red] (p2) -- (k);

  	    \draw[->,>=latex, black, dotted] (ppp2) to[out=-60, in=30] (k);
		\draw[->,>=latex, black, dotted] (pp) to[out=-120, in=30] (k);   

		\draw[->,>=latex, black, dotted] (z) to[out=-120, in=30] (pp);        
		\draw[->,>=latex, black, dotted] (z) to[out=-60, in=30]  (p2);
	\end{tikzpicture} 
}
\end{tabular}
\caption{\label{fig:bin_tree}
		Example of a plane increasing binary tree of size $14$}
\end{figure}
On the left side of Figure~\ref{fig:bin_tree} we present an example of a plane increasing binary tree.
Note that the internal nodes have a left child or a right child or both children.
In particular, the unlabeled subtree rooted at $8$ is the same as the one rooted
at $11$, but they are not the same as the one rooted either at $4$ or at $9$.
The two other structures in the right of the figure are the compacted version of the plane increasing binary tree.
In~\cite[Section 1.3.3]{Drmota09} Drmota exhibits the link between plane increasing binary trees
and binary search trees.

\begin{figure}[h]
	\begin{tabular}{c c}
		\includegraphics[width=0.47\textwidth]{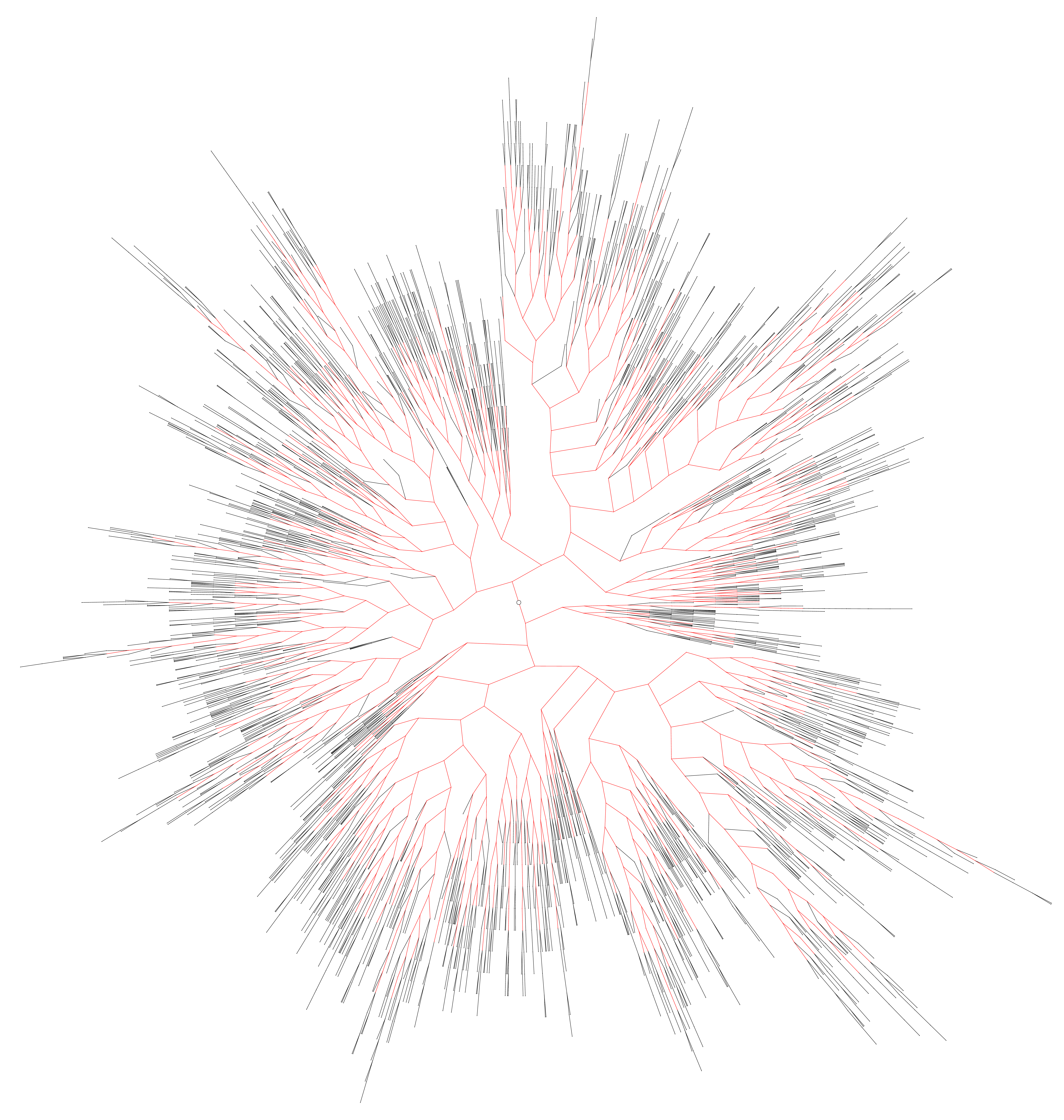}
	&
		\includegraphics[width=0.47\textwidth]{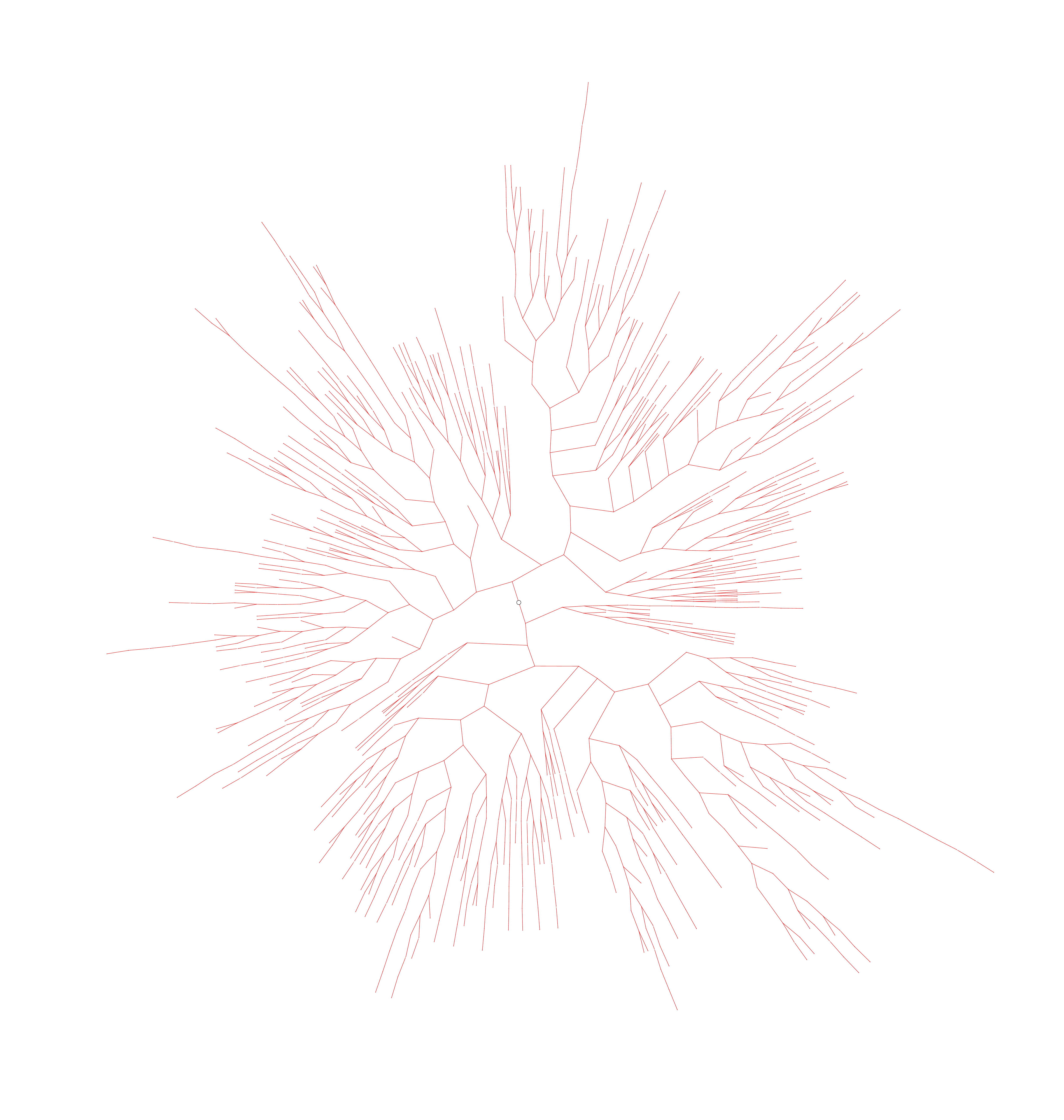}
	\end{tabular}	
    \caption{\label{fig:bin_tree_compact}
		(left) A uniformly sampled (plane) increasing binary tree of size 5,000: Black fringe subtrees are removed
		by the compaction. (right) The red part is of size 1,361.}
\end{figure}
In Figure~\ref{fig:bin_tree_compact} we have represented on the left side a plane increasing binary tree structure
containing $5,000$ nodes. It has been uniformly sampled among all trees with the same size.
The original root of the tree is represented using a small circle $\circ$. On the right side,
we have depicted the nodes that are kept after the compaction of the latter tree. Only $1,361$ nodes
remain.

By using the symbolic method~\cite{flajolet2009analytic}, the latter specification~\eqref{spec:bin_tree} translates as
\[
	T(z) = \int_{0}^{z} \left( 1+T(v)\right)^2 \dv,
\]
in terms of $T(z)$ the exponential generating function for $\mathcal{T}$.
We can also rewrite it as a differential equation
\[
	T'(z)=\( 1+ T(z) \)^2, \qquad \text{with } T(0)=0.
\]
The equation can be solved such that
\[
	T(z)=\frac{z}{1-z},
\]
with the dominant singularity $\rho=1$.

The exponential generating function $S_t(z)$ of the perturbed class of plane increasing binary trees 
that do not contain the tree shape $t$ (where $t$ is a non-labeled binary tree)
as a fringe subtree, fulfills the equation
\begin{equation}
\label{equ:ODE_St}
	S_t'(z)=(1+S_t(z))^2-P_t'(z) \qquad \qquad \text{with } S_t(0)=0
\end{equation}
where $P_t(z)=\frac{\ell(t)z^{|t|}}{|t|!}$ and $\ell(t)$ denotes the number of ways to increasingly label
the plane binary tree $t$. The quantity $\ell(t)$ is also called the hook length of $t$ and 
it is well known that $\ell(t)$ equals $|t|!$ divided by the product of the sizes of
all fringe subtrees of $t$ (\emph{cf.} e.g.~\cite[p.67]{Knuth98} or~\cite{BGP16}). 
We first start with a lemma establishing an upper bound for the normalized hook length.

\begin{lem}
\label{lem:hook}
	Let $t$ be a binary tree of size $k$. By defining the weight of the tree $t$ as
	$w(t):=\frac{\ell(t)}{k!}$, where $\ell(t)$ denotes the hook length of $t$, we have
	\[
		w(t) \leq \frac{1}{2^{k-2}}.
	\]
\end{lem}
\begin{proof}[Key ideas of the proof]
Transforming the hook length formula into a recursive relation we prove
\[
w(t) = 
\begin{cases}
\frac1k w(t') & \text{ if the root of $t$ has one child $t'$} \\
\frac{1}{k} w(t_\ell) w(t_r) & \text{ if the root of $t$ has the two children $t_\ell$ and $t_r$}.
\end{cases}
\]
Set $w_n:=\max_{t\in\mathcal P_n} w(t)$. 
Then compute the first values $w_1$ up to $w_7$, which confirms the claim, and then
proceed by induction for $k\ge 8$. 
\end{proof}

Finally, note that the term by term inverse of the sequence $(w_n)_{n\ge 0}$ 
corresponds to the sequence stored as
\texttt{OEIS A132862}\footnote{This corresponds to the reference of this
  sequence in Sloane's Online Encyclopedia of: Integer
	Sequences \url{www.oeis.org}.}.

By the same combinatorial argument as in the previous section we know that $S_t(z)$ 
has a unique dominant singularity $\tilde{\rho}_t$, which is greater than the dominant 
singularity $\rho=1$ of $T(z)$. Thus, we set again $\tilde{\rho}_t=\rho(1+\epsilon_t)=1+\epsilon_t$.
Since \eqref{equ:ODE_St} is a Riccati differential equation (\emph{cf.}~\cite{Ince44} for 
a background on Riccati equations),
we use the ansatz $S_t(z)=\frac{-u'(z)}{u(z)}$ to get the transformed equation
\begin{equation}
\label{equ:ode_u}
	u''(z)-2  u'(z)+(1-w(t)kz^{k-1})  u(z) = 0,
\end{equation}
where we use the same abbreviations as in the previous section, namely $k:=|t|$ and
$w(t):=\frac{\ell(t)}{k!}$. Note that the condition $S_t(0)=0$ implies $u'(0)=0$ and $u(0)\neq 0$. 

The singularities of a function $u(z)$ solving a linear differential equation (with polynomial coefficients)
are given by the singularities of the coefficient of the highest derivative,
\emph{i.e.}, in our case the coefficient of $u''(z)$, which is 1. 
The reader can refer to Miller~\cite{Miller06} for details.
Thus, we can conclude that $u(z)$ is an entire function. As a direct consequence we know that the 
singularities of $S_t(z)$ are given by the zeros of $u(z)$ (that are not zeros of $u'(z)$) and are therefore poles.
More precisely the dominant singularity $\tilde{\rho}_t$ must be a simple pole for $S_t(z)$,
since for $u(z)=(\tilde{\rho}_t-z)^lv(z)$, (such that $\rho$ is not a zero of $v(z)$), it follows that
$u'(z) = -(\tilde{\rho}_t-z)^{l-1}v(z)+(\tilde{\rho}_t-z)^lv'(z)$. Thus
\[
	S_t(z)=\frac{l}{\tilde{\rho}_t-z}-\frac{v'(z)}{v(z)},
\]
which implies 
\begin{equation*}
	S_t(z) \underset{z \rightarrow \tilde{\rho}_t}\sim \frac{l / \tilde{\rho}_t}{1- z / \tilde{\rho}_t}.
\end{equation*}
Taking the derivative we get
$S_t'(z) \sim \frac{1}{\tilde{\rho}_t^2}\frac{l}{\left( 1- z / \tilde{\rho}_t \right)^2}$.
Plugging in the asymptotic expressions for $S_t$ and $S_t'$ in the original differential equation~\eqref{equ:ODE_St} we get
\[
	\frac{1}{\tilde{\rho}_t^2}\frac{l}{\left( 1- \frac{z}{\tilde{\rho}_t} \right)^2}
		 \underset{z \rightarrow \tilde{\rho}_t}\sim \left( 1 + \frac{l / \tilde{\rho}_t}{1- \frac{z}{\tilde{\rho}_t}} \right)^2,
\]
since the monomial $P_t$ is analytic in $\tilde{\rho}_t$. 
Comparing the main coefficients yields $l= 1$, and thus $\tilde{\rho}_t$ is a simple zero of the function $u(z)$ and 
\[
	S_t(z) \underset{z \rightarrow \tilde{\rho}_t}\sim \frac{1}{\tilde{\rho}_t-z}.
\]

\subsubsection*{How to proceed}
As in the previous section, we have a singularity $\tilde{\rho}_t = 1+\epsilon_t$ with $\epsilon_t>0$ depending on $t$, or $k$. 
In order to get results on the average size of the compacted tree of a random increasing binary tree
we proceed similarly to the recursive tree case. Lemma~\ref{lem:asympepsilon_binary} gives an asymptotic 
expression for $\tilde{\rho}_t$ that quantifies its dependence on $t$, when the size $k$
of the ``forbidden'' tree tends to infinity. 

As a next step, Lemma~\ref{lem:nosing} shows that $S_t(z)$ has a unique dominant singularity $\tilde{\rho}_t$
on the circle of convergence, which is used in Proposition~\ref{prop:fracsttasymp} to obtain
the asymptotic behavior of the coefficients of the generating function $S_t(z)$. 

Again, the average size of a compacted tree can be represented as a sum over the forbidden trees, 
where we distinguish between the two cases whether the size of the trees is smaller or larger
than $\ln n$ in order to get an upper bound (see Propositions~\ref{prop:estimate_sum1} 
and~\ref{prop:thirdsum_binary}).

\bigskip
We start from the equation $u''(z)-2  u'(z)+(1-w(t)kz^{k-1})  u(z) = 0$
with the initial conditions $u(0)=\gamma$, and $u'(0)=0$. 
The value $\gamma$ can be chosen arbitrarily, as $S_t(z)=\gamma u'(z) / (\gamma u(z))$,
and thus, $\gamma$ cancels.	For simplification reasons in the following
we choose $u(0)=-1$ together with the initial condition $u'(0)=0$.
\begin{lem}
\label{lem:u_value}
	The function $u(z)$ defined by the differential equation~\eqref{equ:ode_u} and
	the initial conditions $u(0)=-1$ and $u'(0)=0$ satisfies
	\[
		u(z)= z e^z \sum_{m \geq 0} \left( \frac{w(t)k}{(k+1)^2} \right)^m 
		\frac{1}{m!  (m+\alpha)_m} z^{(k+1)m} 
		- e^z \sum_{m \geq 0} \left( \frac{w(t)k}{(k+1)^2} \right)^m \frac{1}{m!  (m-\alpha)_m} z^{(k+1)m},
	\]
	where $(x)_m$ denotes the falling factorials $(x)_m=x(x-1)\cdots(x-m+1)$
	and $\alpha = 1/(k+1)$.
\end{lem}
Recall (for details refer to the book of Bender and Orszag~\cite{BO99}) that the solutions $y(z)$
of the ordinary differential equation 
\[
	z^2  y''(z) + z  y'(z) + (z^2 - \alpha^2)  y(z) = 0,
\]
with $\alpha$ not being an integer are linear combinations of the 
Bessel functions $J_{\alpha}(z)$ and $Y_{\alpha}(z)$.
Some modifications on \eqref{equ:ode_u} let us exhibit the combination of Bessel functions
that yields the result of Lemma~\ref{lem:u_value}.
%
%
%
\begin{proof}[Proof (sketch)]
Substituting $y(z):= u(z) \cdot \exp(-z) / \sqrt{z}$ and then $x:= \left(
\frac{k+1}{2\sqrt{-w(t)k}}z \right)^{2/(k+1)}$ transforms the differential equation for $u(z)$
into 
\[
      \beta^2  y''( \beta ) + \beta  y'(\beta) + \left( \beta^2 - \frac{1}{(k+1)^2}
      \right)  y(\beta)= 0,
\]
with $\beta = \frac{2\sqrt{-w(t)k}}{k+1} t^{(k+1)/2}$.  We recognize the Bessel equation and thus
$y(\beta)$ is a linear combination of the Bessel functions 
$J_{\alpha}(\beta)$ and $Y_{\alpha}(\beta)$. 

Due to the relationship between the function $u(z), y(\beta)$ and the Bessel functions, 
we deduce $u(z)$ is a linear combination of the functions $f(z)$ and $\bar{f}(z)$ where 
\[
f(z) = \sqrt{z} \exp(z) J_{\alpha} \left( 2\tilde{\beta} z^{\frac{1}{2\alpha}}\right)
\quad\text{ and }\quad
\bar{f}(z) = \sqrt{z} \exp(z) J_{-\alpha} \left( 2\tilde{\beta} z^{\frac{1}{2\alpha}} \right),
\]
with $\tilde{\beta}:= \frac{\sqrt{-w(t)k}}{k+1}$ and $\alpha:=\frac{1}{k+1}$.

By means of the initial conditions $u(0)=-1$ and $u'(0)=0$ the coefficients of the linear
combination can be computed. Finally, using the well-known power series expansions for
$J_{\alpha}(x)$ and $J_{-\alpha}(x)$ as well as the formula
$\frac{\Gamma(1+\alpha)}{\Gamma(m+1+\alpha)}=\frac{1}{(m+\alpha)_m}$, with 
$(x)_m$ being the falling factorials $(x)_m=x(x-1)\cdots(x-m+1)$, the previously obtained sum of
power series eventually simplifies to the expression in the assertion. 
\end{proof}

We are now ready to analyze the dominant singularity of $S_t(z)$.
\begin{lem}
\label{lem:asympepsilon_binary}
	Let $S_t(z)$ be the generating function of the perturbed combinatorial class of plane 
	increasing binary trees that do not contain the shape $t$ as a subtree (of size $k$).
	With $\tilde{\rho}_t$ denoting the dominant singularity of $S_t(z)$, we get
	\[
		\tilde{\rho}_t = 1 + \epsilon_t \underset{k\rightarrow \infty}\sim  1 + \frac{2w(t)}{k^2},
	\]
	where $w(t)=\frac{\ell(t)}{k!}$ and $\ell(t)$ denotes the hook length of $t$.
\end{lem}

\begin{proof}
For combinatorial reasons we deduced that the equation $u(z)=0$ must have a solution $\tilde{\rho}_t>1$ 
and no smaller positive solution. When $k$ tends to infinity we expect that
$\tilde{\rho}_t=1+\epsilon_t$ tends to $1$, \emph{i.e.} $\epsilon_t$ tends to $0$. 

First observe that $u(0)=-1$ and 
\[
	u\(1+\frac1{k^2}\) = \frac1{k^2} + \Ord{\frac{w(t)}k} > 0, 
\]
as $w(t)$ decays exponentially due to Lemma~\ref{lem:hook}. Thus $\epsilon_t=\Ord{1/k^2}$ and 
plugging $z=1+\epsilon_t$ into $u(z)=0$ gives then 
\[
	\epsilon_t + (1+\epsilon_t)^{k+1} \frac{w(t)k}{(k+1)^2} \left( \frac{1+\epsilon_t}{1+\alpha} 
		- \frac{1}{1-\alpha} \right) = \mathcal{O} \left( \frac{w(t)^2}{k^2} \right).
\]
This implies $\epsilon_t - 2 w(t)/k^2 = \mathcal{O} \left( w(t)^2/k^2 \right)$  
and hence $\epsilon_t \sim 2 w(t)/k^2$, which finishes the proof.
\end{proof}
So, Lemma~\ref{lem:asympepsilon_binary} ensures that for $|t|=k$ tending to infinity
the generating function $S_t(z)$ has a dominant singularity at $\tilde{\rho}_t \sim 1+ 2w(t) / k^2$.
Now we show that in a sufficiently large disk there is no other singularity for $S_t(z)$.
\begin{lem}
\label{lem:nosing}
	Let $\tilde{\rho}_t$ be the dominant singularity of $S_t(z)$. 
	Then, for all $\delta>0$ the following assertion holds: If $k$ is sufficiently large,
	then the generating function $S_t(z)$ does not have any singularity in
	the domain $\tilde{\rho}_t < |z| < 1 + \frac{(1-\delta)\ln(1/w(t)) + \ln k}{k}$.
\end{lem}

\begin{proof}[Proof (sketch)]
First note that the singularities of $S_t(z)$ are exactly the zeros of $u(z)$. Define
$\tilde{u}(z) := u(z)\exp(-z)$ and note that $u(z)$ and $\tilde{u}(z)$ have the same zeros. By
Lemma~\ref{lem:u_value} we can write $\tilde{u}(z) = zF(z)-G(z)$ with
\begin{align*}
	F(z) &= \sum_{m \geq 0} \left( \frac{w(t)k}{(k+1)^2} \right) ^m \frac{1}{m!} \frac{1}{(m+\alpha)_m} z^{(k+1)m}, \quad\text{and}\\
	G(z) &= \sum_{m \geq 0} \left( \frac{w(t)k}{(k+1)^2} \right) ^m \frac{1}{m!} \frac{1}{(m-\alpha)_m} z^{(k+1)m},
\end{align*}
with $\alpha := 1 / (k+1)$. Therefore we get $|F(z)-G(z)|= \Ord{w(t)|z|^{k+1}/k^2}$. 
Now, let us rewrite $\tilde{u}(z)$ as
\begin{align}
\label{equ:u(z)ggf}
\tilde{u}(z)=(z-1)F(z)+F(z)-G(z),
\end{align}
set $|z|=1+\eta$ and perform a distinction of two cases:

{\bf Case 1:} $\eta = \mathcal{O} \left( 1/k \right)$.
This implies $|z|^{k+1} = \Theta(1)$ for $k$ tending to infinity.
Thus $F(z) \sim 1$, $G(z) \sim 1$, and then $F(z)-G(z)\to 0$. 
In view of this, \eqref{equ:u(z)ggf} and $\tilde{u}(z)=0$ imply $z-1 \sim G(z)-F(z)= 
\mathcal{O} \left( w(t)/k^2 \right)$ and thus $|z-1| = \Ord{\tilde{\rho}_t -1}=o(1/k)$. 

But for zeros $z_0$ of $\tilde{u}(z)$ with $|z_0|=1+o\left( 1/k \right)$ we know 
$z_0-1 \sim (2w(t)/k^2)\cdot z_0^k \sim 2w(t)/k^2$, so $z_0^k \sim 1$. 
Hence $z_0 \sim \sqrt[k]{1} = \cos \left( \frac{2\pi}{k} \right) + i \sin \left(
\frac{2\pi}{k} \right)$ which contradicts $z_0-1 \sim 2w(t)/k^2$. 
Thus, the function $\tilde{u}(z)$
has no zeros in the domain $\tilde{\rho}_t < |z| \leq 1+ \mathcal{O} \left( 1/k\right)$.

{\bf Case 2:} $\eta = C_k/k$, with $C_k \leq (1-\delta)\ln \rcp{w(t)}+\ln k$.
In this case we have $|z|^{k+1} \le e^{C_k} = \Ord{k/w(t}^{1-\delta})$, and thus 
$|F(z)-G(z)|=o(1/k)$ and $F \sim 1 + \Ord{w(t)^\delta}$. Using again 
\eqref{equ:u(z)ggf} yields 
\begin{equation} \label{tilde_u}
\tilde{u}(z) = z-1 + o(|z-1|w^\delta)+o(1/k) \sim z-1.
\end{equation} 
	Since $|z|=1+\eta$ we have $|z-1| \geq C_k/k>1/k$ and thus $\tilde{u}(z)$ cannot be zero
in $\tilde{\rho}_t < |z| < 1+ ((1-\delta)\ln \rcp{w(t)}+\ln k)/k$.	\qedhere
\end{proof}

Now we are interested in the ratio $[z^n]S_t(z) / [z^n]T(z)$, which corresponds 
to the probability that a random plane binary tree of size $n$ does not contain the 
binary tree shape $t$ as a fringe subtree.

\begin{prop}
\label{prop:fracsttasymp}
Let $T(z)$ be the generating function of plane binary increasing trees and $S_t(z)$ the generating
function of the perturbed class that has the dominant singularity $\tilde{\rho}_t$. Fix a constant
$L>2$. Then, uniformly for $D\le |t|\le n$ with $D$ independent of $n$ and sufficiently large, 
the following asymptotic relations hold, depending on the magnitude of $w(t)$:
\begin{itemize}
\item If $\ln\rcp{w(t)}\le L k$, then
\[
[z^n]S_t(z) = \tilde{\rho}_t^{-n-1}
\(1+\Ord{\exp\(-\frac nk\cdot \frac{\ln(L+1)}L
\ln\rcp{w(t)}\)}\), \text{ as } \nti.
\]
\item If $\ln\rcp{w(t)}>Lk$, then
\[
[z^n]S_t(z) = \tilde{\rho}_t^{-n-1}
\(1+\Ord{ \ln (k) \exp\(-n \(\ln\((1-\delta)\ln\rcp{w(t)}\)-\ln k\)\)}\),
\text{ as } \nti, 
\]
with arbitrary $\delta>0$.
\end{itemize}
\end{prop}
\brem
In contrast to Proposition~\ref{lem:asympst} there are only two cases. The reason is that we know
from Lemma~\ref{lem:hook} that $\ln\rcp{w(t)}$ cannot be too small. In fact, we have $(k-1)\ln
2\le \ln\rcp{w(t)}\le k\ln k$.
\erem


\begin{proof}[Proof (sketch)]
First, let us remember that $\tilde{\rho}_t$ is a unique zero of the function $u(z)$. Thus, we
can write
\begin{equation}
\label{equ:u(z)v(z)}
	u(z)= \left( 1- \frac{z}{\tilde{\rho}_t} \right) v(z),
\end{equation}
with $v(\tilde{\rho}_t) \neq 0$ and by Lemma \ref{lem:nosing} we additionally know that
$v(z) \neq 0$ in $\tilde{\rho}_t < |z| < 1 + \frac{(1-\delta)\ln(1/w(t))+\ln k}{k}$,
provided that $k$ is sufficiently large. This implies  
\[
	S_t(z) = \frac{1}{\tilde{\rho}_t-z} - \frac{v'(z)}{v(z)}.
\]
And thus,
\begin{equation}
\label{equ:Stasympschritt1}
	[z^n]S_t(z)=\tilde{\rho}_t^{-n-1}-[z^n]\frac{v'(z)}{v(z)} = \tilde{\rho}_t^{-n-1} - (n+1)[z^{n+1}]\ln v(z).
\end{equation}
Now, we estimate the second summand in \eqref{equ:Stasympschritt1}.
First we use a Cauchy integral to write
\begin{equation}\label{equ:intforv}
	n~[z^n]\ln v(z) = \frac{n}{2\pi i} \int_{\mathcal{C}} \frac{\ln v(t)}{t^{n+1}}\dt,
\end{equation}
where the curve $\mathcal{C}$ is described by $|t|=1+ \frac{(1-\delta)\ln(1/w(t))+\ln k}{k}$ 
with some $\delta>0$. The absolute value of the logarithm of $v(z)$ is given by
$| \ln v(z)| = \left| \ln \left( |v(z)| e^{i\arg v(z)} \right) \right| = \left| \ln |v(z)| +  
i\arg(v(z)) \right|$. Furthermore, by \eqref{equ:u(z)v(z)} we have
$|v(z)| = |u(z)|/\left| 1 - z / \tilde{\rho}_t \right|$, 
which can be estimated along $\mathcal{C}$ via
\[
	\frac{|u(z)|}{2+\ln k}\leq |v(z)| \leq \frac{k|u(z)|}{(1-\delta)\ln(1/w(t))}.
\]
Now, we have to estimate $|u(z)|$. Using the expansion in Lemma~\ref{lem:u_value} and estimating,
we find that there is a $\mu>0$ such that 
\begin{align*} 
|u(z)| &\leq e^{|z|} \sum_{m \geq 0} \left( \frac{w(t)}{k} \right)^m \frac{1}{m!} 
\left| \frac{z}{(m+\alpha)_m} - \frac{1}{(m-\alpha)_m} \right| |z|^{(k+1)m} 
\\ 
&\le e^{|z|} \sum_{m \geq 0} w^{\delta m}\frac{2+\mu}{m!(m-\alpha)_m}=\Ord{k}.
\end{align*} 
To get a lower bound, observe that $|u(z)|\ge
e^{-|z|} |\tilde u(z)|\ge |\tilde u(z)|/(ke)$ and by \eqref{tilde_u} we have $\tilde u(z)\sim z-1$.
This yields $|u(z)|\ge (1-\delta)\ln(1/w(t))/(k^2e)=\Omega(\ln(k)/k^2)$.
	
Putting all together, we can estimate the integral~\eqref{equ:intforv} by
\begin{align*}
	n|[z^n]\ln v(z)| &= \Ord{n\ln k \(1+\frac{(1-\delta)\ln(1/w(t))+\ln k}k\)^{-n}} \\
	&= \Ord{n\tilde\rho_t^{-n} \ln k \(1+\frac{(1-\delta)\ln(1/w(t))+\ln k-\delta}k\)^{-n}}\\ 
	&=\Ord{n\tilde\rho_t^{-n} \ln k \(1+\frac{(1-\delta)\ln(1/w(t))}k+\frac{\ln n}n\)^{-n}}.
\end{align*} 
Finally, proceed as at the end of the proof of Proposition~\ref{lem:asympst} to complete the
proof.  
\end{proof}

Now, we split the sum of interest, \emph{i.e.} $\sum_{t \in \mathcal{B}} \mathbb{P} \left[ t
\text{ occurs at subtree of } \tau \right]$, where $\tau$ denotes a plane increasing binary tree
of size $n$ and $\mathcal{B}$ denotes the class of (unlabeled) plane binary trees, analogously as
we did in the previous section for recursive trees.

\begin{rem}
Now our underlying class of tree shapes is the class of plane binary trees and no more the 
class of P\'olya trees. Since the dominant singularity of the generating function 
of binary trees is $1/4$, we use henceforth $\log_4 n$, the logarithm with respect to base $4$.
\end{rem}

\begin{equation}\label{equ:threesums_binary}
	\mathbb{E} \left( X_n \right) = \sum_{\substack{t \in \mathcal{B}_{\leq n}\\ k< \log_4 n}} \left( 1 - \frac{[z^n]S_t(z)}{[z^n]T(z)} \right) +  
	\sum_{\substack{t \in \mathcal{B}_{\leq n}\\ k \geq \log_4 n}} \left( 1 - \frac{[z^n]S_t(z)}{[z^n]T(z)} \right).
\end{equation}
In order to estimate the first sum, we proceed analogously to Proposition~\ref{lem:estimate_sum1}.
\begin{prop}\label{prop:estimate_sum1}
	Let $B(z)$ be the generating function associated to $\mathcal{B}$,
	of (unlabeled) binary trees, whose dominant singularity is~$1/4$.
	Then asymptotically when $n$ tends to infinity we have
	\[
		\sum_{\substack{t \in \mathcal{B}_{\leq n}\\ k < \log_4 n}} \( 1 - \frac{[z^n]S_t(z)}{[z^n]T(z)} \)
			 \underset{n\rightarrow \infty}=  \mathcal{O} \( \frac{n}{\sqrt{(\ln n)^3}} \).
	\]
\end{prop}
\begin{proof}
	A crude estimate gives 
	\begin{align*}
		\sum_{\substack{t \in \mathcal{B}_{\leq n}\\ k< \log_4 n}} \left( 1 - \frac{[z^n]S_t(z)}{[z^n]T(z)} \right) 
			&\leq \sum_{\substack{t \in \mathcal{B}_{\leq n}\\ k < \log_4 n}} 1 
				= \sum_{k < \log_4 n} [z^k]B(z) \underset{n\rightarrow \infty}\sim \frac{1}{1-\frac{1}{4}} [z^{\lfloor \log_4 n \rfloor}]B(z)\\ 
			&\underset{n\rightarrow \infty}= \mathcal{O} \( \frac{\left( \frac{1}{4} \right) ^{- \lfloor \log_4 n \rfloor}}{\sqrt{(\log_4 n)^3}} \).
	\end{align*}
	This is already sufficient, since $\left( \frac{1}{4} \right) ^{- \lfloor \log_4 n \rfloor} \leq n$, which completes the proof.
\end{proof}
Estimating the second sum in \eqref{equ:threesums_binary} is based on some counting arguments,
analogously to the  proof of Proposition~\ref{lem:estimate_sum3} in the previous section. 
However, due to the fewer grafting possibilities for the tree shape $t$ a straight-forward analog
of the proof of Proposition~\ref{lem:estimate_sum3} yields a too crude upper bound. Thus a finer
analysis is needed. 

\begin{prop}\label{prop:thirdsum_binary}
	Let $\mathcal{B}_{\leq n}$ denote the class of binary trees of size at most $n$. Then
	\[
		\sum_{\substack{t \in \mathcal{B}_{\leq n}\\ k \geq \log_4 n}}
			\left( 1 - \frac{[z^n]S_t(z)}{[z^n]T(z)} \right) \underset{n\rightarrow \infty}= \mathcal{O} \left( \frac{n}{\ln n} \right).
	\]
\end{prop}
\begin{proof}
Using a similar counting approach as the one proposed in Proposition~\ref{lem:estimate_sum3},
we obtain for $k=|t|<n$ 
\[
1 - \frac{[z^n]S_t(z)}{[z^n]T(z)} \leq  \frac{1}{n!} \sum_{\nu = 2}^{n-k+1} \binom{n-\nu}{k-1}
\ell(t) (n-k)! (\nu-1).
\]

Let $t$ be the shape that appears in a tree of size $n$ and $\nu$ be the root label of this
occurrence of $t$. Then there at most $\nu-1$ possibilities to attach $t$ to a tree of size $n-k$,
because the node to which $t$ is attached must have a label smaller than $\nu$ and a free place,
as we consider incomplete binary trees here. Moreover, there are $\binom{n-\nu}{k-1}$ ways to
choose the labels for $t$, $\ell(t)$ ways to make a proper labeling on $t$ with the chosen labels,
and $(n-k)!$ trees of size $n-k$ to which $t$ will be attached. As we want an upper bound, we do
not care for multiple counting. 

After simplification we obtain $\ell(t)/k! \cdot (n-k)/(k+1)$, but this is too large to get the
analog of Proposition~\ref{lem:estimate_sum3}. The problem here comes from the fact there are
usually much fewer possibilities to graft $t$, in particular if $|t|$ is small. 
To get a better upper bound, we rely on \cite[Theorem~5]{BG15}, where it is proved that the number
of binary increasing trees of size $i+j$ having exactly $k$ subtrees attached to the head of the
tree (that is the minimal subtree that contains all nodes labeled with the smallest $i$ 
labels) is
\[
	\binom{i+1}{k} \binom{j-1}{k-1} ~i!~j!.
\]
Here we are interested in trees of size $n-k$ containing the first $\nu-1$ labels and having
exactly $r$ available possibilities to graft the tree $t$, thus $\nu-r$ trees are already attached
to the head. According to the above formula the
number of such trees is
\[
	\binom{\nu}{\nu-r} \binom{n-k-\nu}{\nu-r-1}  ~(\nu-1)!~(n-k-\nu+1)!.
\]
If $\nu<n-k+1$ this value is correct for $r$ ranging from 0 to $\nu-1$. 
Otherwise if $\nu=n-k+1$ then $r$ can also be equal to $\nu$ and then the number of
possibilities to attach $t$ to all heads of size $n-k$ is $(n-k+1)!$. So we obtain the better upper
bound 
\begin{align*}
1 - \frac{[z^n]S_t(z)}{[z^n]T(z)} \leq & \frac{1}{n!} \sum_{\nu = 2}^{n-k+1} \binom{n-\nu}{k-1}
\ell(t)   \sum_{r=1}^{\nu-1} r  \binom{\nu}{\nu-r} \binom{n-k-\nu}{\nu-r-1}  (\nu-1)!  (n-k-\nu+1)!\\
& + \ell(t) \frac{(n-k+1)!}{n!}.
\end{align*}
Following~\cite[p. 169]{GKP94} this simplifies to  
\[
\sum_{r=1}^{\nu-1} r  \binom{\nu}{\nu-r} \binom{n-k-\nu}{\nu-r-1} = \nu  \binom{n-k-1}{\nu-2}.
\]
Using this result we deduce
\begin{align*}
1 - \frac{[z^n]S_t(z)}{[z^n]T(z)} &\leq \frac{\ell(t)}{n!} 
\sum_{\nu = 2}^{n-k+1} \frac{\nu!  (n-\nu)!  (n-k-1)!}{(k-1)!  (\nu-2)!  (n-k-\nu+1)!} + \ell(t) \frac{(n-k+1)!}{n!}\\
&= \ell(t) \frac{(n-k-1)!}{n!} 	\sum_{\nu = 2}^{n-k+1} \nu(\nu-1)  \binom{n-\nu}{k-1} + \ell(t) \frac{(n-k+1)!}{n!}.
\end{align*}			
Again using~\cite[p. 169]{GKP94}, we further simplify and get 
\[
\sum_{\nu = 2}^{n-k+1} \nu(\nu-1)  \binom{n-\nu}{k-1} = 2 \sum_{\nu = 2}^{n-k+1} \binom{\nu}{2}  \binom{n-\nu}{k-1} = 2 \binom{n+1}{k+2}.
\]
We thus conclude
\begin{align*}
1 - \frac{[z^n]S_t(z)}{[z^n]T(z)} &\leq \frac{\ell(t)}{k!} \left( \frac{2(n+1)}{(k+2)(k+1)} + \frac{k}{n}\frac{k-1}{n-1}\dots \frac{1}{n-k+1}\right) \\
& \leq \frac{\ell(t)}{k!} \left( \frac{2(n+1)}{(k+2)(k+1)} + \frac{1}{n-k+1}\right).  
\end{align*}

Furthermore, for $|t|=n$, we have
\[
1 - \frac{[z^n]S_t(z)}{[z^n]T(z)} = \frac{\ell(t)}{n!}.
\]	
Finally, we finish the proof like in Proposition~\ref{lem:estimate_sum3}
and get the stated result.
\end{proof}

\begin{thm}
\label{thm:bounded_exp}
	Let $X_n$ be the size of the compacted tree corresponding to a 
	random binary increasing tree of size $n$. Then
		\[
			\mathbb{E} \left( X_n \right) \underset{n\rightarrow \infty}= \mathcal{O} \left( \frac{n}{\ln n} \right).
		\]
\end{thm}
\begin{proof}
The result follows directly by combining the previous propositions.
\end{proof}

Recall that this result has already been shown in~\cite{FGM97,Devroye98}, even with $\Theta$
instead of big-$\mathcal{O}$.  Other proofs were presented as well, see~\cite{BL18,BW20}.

To get a crude lower bound for the number of non-isomorphic subtree shapes in a random increasing
binary tree, we may proceed as in the case of recursive trees. Indeed, the uniform asymptotics
given in Proposition~\ref{prop:fracsttasymp} enable us to derive a lower bound for for the dominant
singularity: $\tilde\rho_t> 1+ w(t)/k^2$ (\textit{cf.} Corollary~\ref{cor:coeff_lower_bound}).
With this bound we can perform all the steps of the proof of Proposition~\ref{thm:recursive_omega}
and get the lower bound $\Omega(\sqrt n\,)$.

Likewise, if the analog of Conjecture~\ref{c1} for plane binary increasing trees were true, then
we would be able to redo the proof of Theorem~\ref{thetabound} to obtain the lower bound
$\Omega(n/\ln n)$. Unfortunately, even the better knowledge on $w(t)$ given by
Lemma~\ref{lem:hook} is not sufficient to show the analog of Conjecture~\ref{c1}.

%% file: sections/datastruct.tex
\section{A compressed data structure}
	\label{sec:data}

The probability model induced by plane increasing binary trees
is the classical permutation model of \emph{binary search trees} (or \textsc{bst}).
Thus the typical shape of a uniformly sampled plane increasing binary tree consisting of 
$n$ internal nodes corresponds to the typical shape of a binary search tree 
built using a uniform random permutation of $n$ elements.
See Drmota~\cite[Section 1.3.3]{Drmota09} for details about the latter correspondence.
Thus the tree structure of a typical \textsc{bst} has the properties we have found out 
in the previous section. In particular, by removing the information stored in the nodes 
the typical compaction of the tree
gives a compacted structure consisting of $\mathcal{O}(n / \ln n)$ nodes (on average).

%
%
%
%


Throughout this section, we aim at designing a new data structure based on the tree structure induced by the
compaction of a \textsc{bst} associated to some extra information in the nodes and the edges 
in order to keep all the information (the integer values) from the original \textsc{bst}.
And of course we must be able to retrieve information efficiently, as in \textsc{bst}s. Our approach
is supported with a python prototype and the experiments are obtained through this 
implementation.

The \textsc{bst} built for example on the permutation $(4, 8, 6, 2, 9, 1, 3, 7, 5)$ 
is represented with the classical tree structure in the left-hand side of Figure~\ref{fig:BST}.
This example will be used as an illustration throughout the whole section.
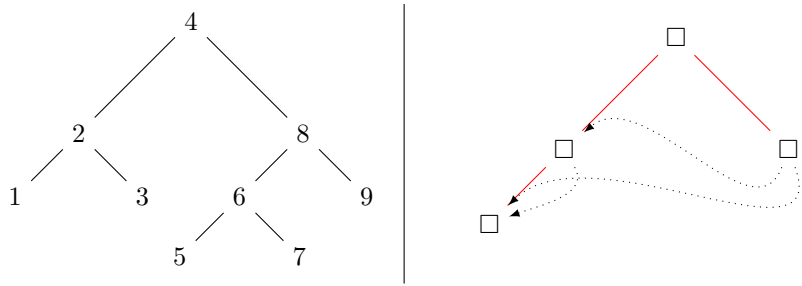
\begin{figure}[h]
	\begin{tabular}{c|c}
	\begin{tikzpicture}[node distance=32pt]
 		    \node (a) {$4$};
 		    \node[below left of=a, node distance=60pt] (b) {$2$};
 		    \draw (a) -- (b);
 		    \node[below right of=a, node distance=60pt] (bb) {$8$};
 		    \draw (a) -- (bb);
 		    \node[below left of=bb, node distance=34pt] (ee) {$6$};
 		    \draw (bb) -- (ee);
 		    \node[below right of=bb, node distance=34pt] (eeee) {$9$};
 		    \draw (bb) -- (eeee);		    
    		\node[below right of=ee, node distance=32pt] (eee) {$7$};
 		    \draw (ee) -- (eee);
 		    \node[below left of=ee, node distance=32pt] (eee) {$5$};
 		    \draw (eee) -- (ee);
 		    \node[below left of=b, node distance=34pt] (c) {$1$};
 		    \draw (b) -- (c);
 		    \node[below right of=b, node distance=34pt] (d) {$3$};
 		    \draw (b) -- (d);
	\end{tikzpicture}
	\qquad
	&
	\qquad
    \begin{tikzpicture}[node distance=32pt]
            \node (a) {$\square$};
            \node[below left of=a, node distance=60pt] (b) {$\square$};
            \draw[red] (a) -- (b);
            \node[below right of=a, node distance=60pt] (bb) {$\square$};
            \draw[red] (a) -- (bb);
            \draw[->,>=latex, black,dotted] (bb) to[out=-110,in=40] (b);
            
            \node[below left of=b, node distance=40pt] (c) {$\square$};
            \draw[red] (b) -- (c);
            \draw[->,>=latex, black, dotted] (bb) to[out=-70,in=44] (c);		   
            \draw[->,>=latex, black, dotted] (b) to[out=-60,in=20] (c);
            
            
            \node[below right of=b, node distance = 22pt] (ee) {};
    \end{tikzpicture}
   	\end{tabular}	
    \caption{\label{fig:BST}
		(left) A \textsc{bst} built e.g. on $(4, 8, 6, 2, 9, 1, 3, 7, 5)$; (right) The compacted tree structure
		associated to the \textsc{bst}}
\end{figure}
In order to compress the tree structure, first the node labels must be removed, as presented
before. Thus by using a compaction through a postorder traversal of the tree, the example becomes
the tree structure presented in the right-hand side Figure~\ref{fig:BST}.
By adding the values stored in the original \textsc{bst} we get the tree of Figure~\ref{fig:BST2}.
\begin{figure}[h]
    \begin{tikzpicture}[node distance=32pt]
        \node (a) {$4$};
		\node[above left of=a, node distance=10pt, scale=0.5] (ta) {\textcircled{\raisebox{-1pt}{\textcolor{blue}{9}}}};
        \node[below left of=a, node distance=60pt] (b) {$2$};
		\node[above left of=b, node distance=10pt, scale=0.5] (tb) {\textcircled{\raisebox{-1pt}{\textcolor{blue}{3}}}};
        \draw[red] (a) -- (b);
        \node[below right of=a, node distance=60pt] (bb) {$8$};
		\node[above right of=bb, node distance=10pt, scale=0.5] (tbb) {\textcircled{\raisebox{-1pt}{\textcolor{blue}{5}}}};
        \draw[red] (a) -- (bb);
        \draw[->,>=latex, black, dotted] (bb) to[out=-110,in=40] node [near start, fill=white, scale=0.6]{$[6,5,7]$} (b);
        
        \node[below left of=b, node distance=40pt] (c) {$1$};
		\node[above left of=c, node distance=10pt, scale=0.5] (tc) {\textcircled{\raisebox{-1pt}{\textcolor{blue}{1}}}};
        \draw[red] (b) -- (c);
        \draw[->,>=latex, black, dotted] (bb) to[out=-70,in=44] node [very near start, fill=white, scale=0.6]{$[9]$} (c);		   
        \draw[->,>=latex, black, dotted] (b) to[out=-60,in=20] node [midway, fill=white, scale=0.6]{$[3]$} (c);
    \end{tikzpicture}
    \caption{\label{fig:BST2}
		Labeled compacted structure associated to the original \textsc{bst}}
\end{figure}
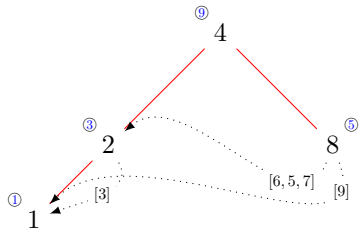
When a substructure has been removed through the compaction process, then in addition to the red pointer,
the list of the labels, obtained through a \emph{preorder traversal} of the substructure is stored. The latter, associated
to the size of the substructures, depicted with the circled blue values,
allows efficient searching. Let us present an example. We would like to know if $7$ is stored in
the structure. $7$ is larger than $4$, thus from the root we take the right edge to reach $8$. 
The value we are looking for is smaller than $8$. We take the left black pointer, and take also in consideration
the list $L:= [6, 5, 7]$. We define an index $i=0$
corresponding to the actual index in the list we are interested in. Using the pointer, we reach $2$ that 
corresponds in fact to $L[0]=6$. Since $7$ is larger than $6$, we must follow the right child of $2$, 
thus the new index is $i := i+2$ (the list stores the values obtained through the preorder traversal),
the constant $2$ is the size of the left subtree attached to $2$
plus $1$ for the node labeled by $2$. Now $L[2]=7$, we have reached the value we were interested in.

\begin{prop}
	In the compacted \textsc{bst} containing $n$ values, the search complexity is the same as in the \textsc{bst} 
	with respect to the number of value comparisons. There may be an extra-cost corresponding to the number of additions (related to the index)
	to traverse a list. The number of additions is at most equal to the number of comparisons to search for the value.
\end{prop}
\begin{proof}
	The number of value comparisons is exactly the same in the compacted structure as in the original \textsc{bst}.
	In fact, we just share the identical unlabeled tree structure, thus the number of comparisons does not change.
	For the same reason, if we must search inside a list associated to a black pointer, then, for each comparison
	there is one addition to shift inside the list.
\end{proof}

In the following Figure~\ref{fig:complex} we have represented two experiments through our python prototype.
On the left-hand side we are interested in the compaction ratio between the compressed data structure and the original \textsc{bst}.
Here we are interested in the whole size needed in memory. In particular the size of the integer values is counted
but further the data structure size itself is important. It is this latter that is in fact compressed:
in the \textsc{bst} many pointers are needed to reach the nodes of the tree. Many pointers and nodes
are replaced in the compressed data by lists of integers that need much less memory in practice.
In the figure, in the abscissa we represent the number of integers stored in the data structures; and in the ordinate, we compute
the ratio between the size in memory of the compressed data structure in front of the size of its corresponding \textsc{bst}.
Each dot corresponds to one sample, and the green curve is the average value among all samples.
The experiments are starting with 250 integer values up to 20,000 with steps every 250 values, and for each size we have used 
30 uniformly sampled \textsc{bst}s. We observe that even for small \textsc{bst}s, the compression ratio is very interesting, smaller than 0.5.
Further we remark that the green curve looks like the theoretical result: it is very close to a function $x \mapsto \alpha / \ln x$
for a given $\alpha$.

\begin{figure}[h]
	\begin{tabular}{c c}
		\includegraphics[scale=0.44]{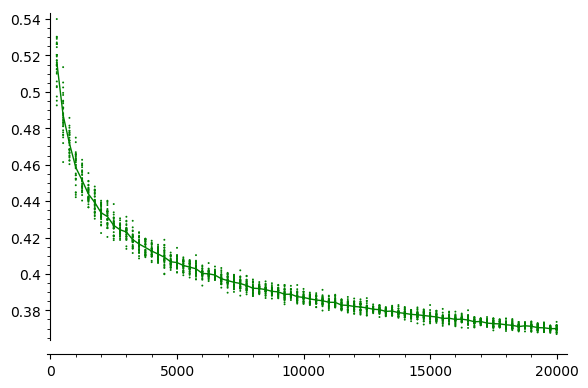}
	\qquad
	&
	\qquad
    	\includegraphics[scale=0.44]{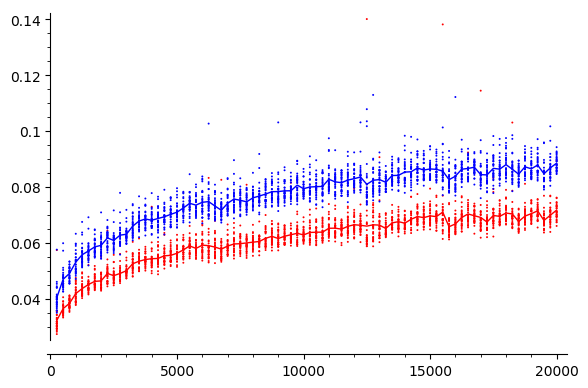}
   	\end{tabular}	
    \caption{\label{fig:complex}
		(left) Experimental compression ratio; (right) Experimental search time comparison}
\end{figure}

On the right-hand side of Figure~\ref{fig:complex}, for the same set of \textsc{bst}s and associated compacted structures,
we search for 1,000 randomly sampled values
present in the two structures. Each red dot is the average time, in milliseconds, (among the 1,000 searches)
for finding the value inside the \textsc{bst}, and the blue point is the analogous time for the search in the compressed structure.
For both complexity measures (number of comparisons or of arithmetic additions) the
average complexity stays of the same order $O(\ln n)$ as for the original \textsc{bst},
as we see it in the figure.
By computing the ratio of the blue values and the red values, the mean seems oscillating around $1.25$
for the whole range of sampled structures.

Let us conclude this section with the following remark. The point of view we have chosen is to build first the \textsc{bst}
and then, once the insertion and deletion process is done, we convert the \textsc{bst} into a compressed data structure
that is used only for searching.
We could develop a prototype data structure that manages insertion in deletion but the efficiency would probably be much
less than the one of \textsc{bst}, because of the substructure recognition problem.

%% file: sections/concl.tex
\section{Conclusion}
	\label{sec:concl}

We showed for two exemplary families of increasing trees that the size of the compacted 
tree is smaller than for simply generated trees. This was done for recursive trees and plane
increasing binary trees. Though the result for the latter family was already known (and even with
better lower bound), we presented a new proof here and showed that our approach might work for
more classes of increasing trees.  

More precisely, we proved that the compacted tree belonging to a random recursive or increasing
binary tree of size $n$ is on average of size $\Omega (\sqrt{n} )$ and $\mathcal{O}(n/\ln n)$.
Numerical simulations on recursive trees suggest that this upper bound is already sharp,
\emph{i.e.}, that the size of the compacted tree is $\Theta(n/\ln)$. For the binary case that
was already shown with other methods.

However, in order to prove this conjecture, one has to find the distribution of the weights
$w(t)$, which turns out to be a very challenging task, especially in case of non-plane trees due
to the appearance of automorphisms. However, we could formulate a simple to state condition under
which we can prove the sharpness of the lower bound.  Thus, obtaining the (maximum) number of
labelings of non-plane trees of a given size is still work in progress, with the aim to improve
the lower bounds such that we can show the $\Theta$-result.  Furthermore, we conjecture that on
average the compacted tree is of size $\Theta \left( \frac{n}{\ln n} \right)$ for all classes of
increasing trees.

We explain the choice of the two classes of increasing trees, that were investigated within this
paper. The reason to choose recursive trees and increasing binary trees was that for these two
classes our computer algebra system is able to solve the differential equation defining $S_t(z)$,
although in case of increasing binary trees the solution is already more complicated and involves
some Bessel functions. On the other hand, this makes it easier, as we could then deal with
explicit expansions.  However, in case of the third prominent class of increasing trees,
\textsc{port}s (plane oriented recursive trees), we did not get any explicit solution for the
analogous of $S_t(z)$; thus this case is still an open question.

As a final note, remember the way we have compacted the \textsc{bst}s in the last section. Using
a pointer to describe the erased fringe subtree and the list of the labels in a specific traversal
(labels that must be kept in the compacted tree), we are able to search in the compacted structure
efficiently. But more generally, the way we have compacted the tree can be used for all possible
tree structures.  In the original paper~\cite{FSS90} by Flajolet~\emph{et al.}, the authors
compact only identical fringe subtrees in simply generated trees.  We focus on the tree structure
and its compaction as well, but the probability model on the tree shapes is a different one,
induced by the labeling. Moreover, we use a different additional information management in order
to cope with labels and could there extend the compaction to labeled tree models. It is desirable
to study other natural labeled tree classes and the resulting compaction ratio.



%
%